\documentclass[journal]{IEEEtran}

\usepackage{xcolor}
\usepackage[font=small]{caption}
\usepackage{amsmath,amsthm,amssymb,amsfonts}
\usepackage[normalem]{ulem}   
\usepackage{bm}
\usepackage{algorithm} 
\usepackage{algorithmic}
\usepackage{comment}
\usepackage[hidelinks]{hyperref}
\usepackage{tabularx}
\usepackage{booktabs}
\usepackage{bbold}
\usepackage{mathtools}
\usepackage{tcolorbox}
\usepackage{bbm}
\usepackage{todonotes}
\usepackage{tensor}
\usepackage{caption}
\usepackage{subcaption}
\usepackage{graphicx}
\usepackage{enumitem}
\setlist{nolistsep,leftmargin=*,labelindent=0pt}
\usepackage{threeparttable}

\newcommand{\tr}{\mathsf{ T}}

\newtheorem{theorem}{Theorem}
\newtheorem{definition}{Definition}

\newtheorem{remark}{Remark}
\newtheorem{lemma}{Lemma}

\newtheorem{proposition}{Proposition}
\newtheorem{example}{Example}

\newtheorem{assumption}{Assumption}

\DeclareMathOperator*{\argmin}{\arg\!\min}

\definecolor{moccasin}{rgb}{0.98, 0.92, 0.84}
\newtcolorbox{mybox}{colback=moccasin,
colframe=moccasin}

\usepackage[nameinlink]{cleveref}
\crefname{equation}{}{}
\crefname{theorem}{Theorem}{Theorems}
\crefname{corollary}{Corollary}{Corollaries}
\crefname{example}{Example}{Examples}
\crefname{assumption}{Assumption}{Assumptions}
\crefname{lemma}{Lemma}{Lemmas}
\crefname{proposition}{Proposition}{Propositions}
\crefname{figure}{Figure}{Figures}
\crefname{table}{Table}{Tables}
\crefname{section}{Section}{Sections}
\crefname{appendix}{Appendix}{Appendices}
\Crefname{equation}{}{}
\Crefname{theorem}{Theorem}{Theorems}
\Crefname{corollary}{Corollary}{Corollaries}
\Crefname{example}{Example}{Examples}
\Crefname{lemma}{Lemma}{Lemma}
\Crefname{proposition}{Proposition}{Propositions}
\Crefname{figure}{Figure}{Figures}
\Crefname{table}{Table}{Tables}
\Crefname{section}{Section}{Sections}
\Crefname{appendix}{Appendix}{Appendices}

\newcommand{\CAP}{\texttt{CAP}}

\newcommand{\col}{\textnormal{col}}

\newcommand{\z}{\textnormal{z}}
\newcommand{\K}{\textnormal{K}}
\newcommand{\D}{\textnormal{d}}

\allowdisplaybreaks
\ifCLASSINFOpdf
 
\else
\fi

\makeatletter
\@ifpackageloaded{cite}{}{\usepackage{cite}}
\makeatother

\begin{document}

\title{On the Existence of Koopman Linear
Embeddings for Controlled Nonlinear Systems}

\author{
\thanks{}
}

\author{Xu Shang \quad Masih Haseli \quad Jorge Cort\'es \quad Yang Zheng \vspace{-2mm}
\thanks{{This work is supported by NSF CMMI 2320697, NSF CAREER 2340713, and ONR Award N00014-23-1-2353.} }
\thanks{X. Shang, J. Cort\'es, and  Y. Zheng are with the Contextual Robotics Institute, UC San Diego, \{x3shang,
cortes,zhengy\}@ucsd.edu. M. Haseli is with the Department of Computing and Mathematical Sciences, California Institute of Technology, mhaseli@caltech.edu.} }

\markboth{Journal of \LaTeX\ Class Files,~Vol.~14, No.~8, August~2015}%
{Shell \MakeLowercase{\textit{et al.}}: Bare Demo of IEEEtran.cls for IEEE Journals}


\maketitle

\begin{abstract}
Koopman linear representations have become a popular tool for control design of nonlinear systems, yet it remains unclear when such representations are exact. In this paper, we establish sufficient and necessary conditions under which a controlled nonlinear system admits an exact Koopman linear representation in both infinite-dimensional and finite-dimensional settings. We show that such a system must be transformable into a special control-affine preserved (\CAP{}) structure, which enforces affine dependence of the state on the control input and isolates all nonlinearities into an autonomous subsystem. We further prove that a finite-dimensional representation additionally requires this autonomous subsystem must itself admit a finite-dimensional Koopman linear model with a sufficiently-rich Koopman invariant subspace. Finally, we introduce a symbolic procedure to determine whether a given controlled nonlinear system admits the \CAP{} structure, thereby elucidating whether Koopman approximation errors arise from intrinsic system dynamics or from the choice of lifting~functions.
\end{abstract}

\IEEEpeerreviewmaketitle

\vspace{-3mm}
\section{Introduction} \label{sec:introduction}
\IEEEPARstart{M}odern dynamical systems are increasingly complex, and their nonlinearities pose significant challenges for analysis and control. 
While controllers can be designed directly from nonlinear models \cite{khalil2002nonlinear}, such methods are~often~computationally demanding, and it can be difficult to obtain~performance guarantees, as in  nonlinear model predictive control (MPC) \cite{rawlings2017model}. These challenges motivate growing interest in linear, or approximately linear, representations of nonlinear systems, which enable the use of mature linear control~techniques and computationally efficient linear MPC~schemes.

One prominent approach for obtaining such linear representations is based on Koopman operator theory \cite{koopman1931hamiltonian}, which lifts nonlinear dynamics into an infinite-dimensional space~where the evolution becomes linear. Koopman-based methods have shown promising performance across a range of practical~applications \cite{haggerty2023control, mamakoukas2021derivative,shi2026koopman,Shang2025KoopmanMPC}. Koopman operator theory was originally developed for autonomous systems. It provides a linear,~infinite-dimensional operator that describes the evolution of observable functions along system trajectories \cite{koopman1931hamiltonian}. This linear representation enables the analysis of complex nonlinear dynamics using tools from linear systems theory and has led to significant developments in areas such as stability analysis \cite{mauroy2016global}, fluid dynamics \cite{mezic2013analysis}, and hybrid systems \cite{katayama2025koopman}; see \cite{mezic2021koopman} for an overview. 

Beyond its infinite-dimensional nature, another fundamental limitation of the classical Koopman operator is its reliance on system autonomy. When control inputs are present, the~state evolution depends explicitly on external signals, and the Koopman operator is no longer applicable in its original form. Several operator-theoretic extensions have been proposed to incorporate control inputs. In particular, the work~\cite{korda2018linear} augments the state with the space of infinite control sequences, thereby embedding the controlled system into an autonomous one. Another work~\cite{haseli2023modeling} introduces a notion of Koopman control family (KCF), which characterizes controlled dynamics via a collection of Koopman operators corresponding to constant inputs. While these extensions provide a conceptual framework for analyzing controlled nonlinear systems, practical controller synthesis cannot be carried out directly at the operator level. 

In practice, one seeks \textit{finite-dimensional} Koopman-based models that approximate the action of the Koopman operator on a chosen set of lifting functions (\emph{i.e.}, observables). This approach leads to explicit state-space models in a high-dimensional lifted space at the cost of incurring modeling errors. \textcolor{black}{For autonomous systems, approximating the Koopman operator with a finite-dimensional linear model is natural, as an exact infinite-dimensional linear representation is guaranteed to exist \cite{mauroy2020koopman}. Conditions for the existence of exact finite-dimensional Koopman linear embeddings have been investigated in \cite{kvalheim2026linearizability,liu2025properties}. We can construct exact finite-dimensional Koopman linear models for systems with specific structure in both continuous time \cite{jungers2019non,iacob2023finite} and discrete time
\cite{wang2023computation}.} For controlled systems, however, exact Koopman linear models generally do not exist, even in infinite dimensions \cite{iacob2024koopman}. This limitation has motivated the development of alternative Koopman-based model classes. For example, Koopman linear time-varying models are introduced in \cite{iacob2024koopman}, with exact formulations derived in \cite{iacob2023finite,iacob2025exact} for polynomial systems and interconnections of linear and nonlinear blocks, respectively. Koopman bilinear models have been proposed for control-affine systems \cite{goswami2021bilinearization,strasser2023control}. Another perspective treats the input as a sequence of constant values, thereby reducing the controlled system to a switching family of autonomous systems \cite{sootla2017pulse,peitz2019koopman}. More recently, as mentioned above, the KCF framework \cite{haseli2023modeling} provides an input-state separable representation that includes Koopman linear and bilinear models as special cases.

Despite the availability of richer model classes, the \textit{finite-dimensional} Koopman linear model remains a widely used representation for controlled systems due to its simplicity and its direct compatibility with established linear control techniques. For example, Koopman linear quadratic regulator (LQR) \cite{brunton2016koopman} and Koopman MPC \cite{korda2018linear} are developed for nonlinear systems admitting exact finite-dimensional Koopman linear models, with closed-loop exponential stability guarantees established in \cite{shang2025exponential}. \textcolor{black}{As mentioned above, general controlled nonlinear systems do not admit an exact Koopman linear representation. This has motivated recent theoretical work on understanding and mitigating the impact of modeling error. When the approximation error is sufficiently small, local exponential stability of the Koopman LQR is shown in \cite{shang2025exponential}. For Koopman MPC, input-to-state stability and local exponential stability are established in \cite{de2024koopman,shang2025exponential}, via suitable adaptations of linear MPC. In addition, constraint satisfaction in the MPC framework under model mismatch is addressed in \cite{zhang2022robust} and~\cite{mamakoukas2022robust} by enforcing more conservative surrogate constraints in the lifted state space. For extended Koopman-based models, robust feedback design and MPC stability guarantees under approximation errors have been developed in \cite{strasser2025koopman,bold2025data,schimperna2026data}. From a practical standpoint, Koopman MPC and its variants have been widely applied in robotic manipulations \cite{zhang2023online}, ground robots \cite{vsvec2023predictive}, and soft robots~\cite{haggerty2023control}. A comprehensive overview of these developments can be found in \cite{shi2026koopman}.}

Despite these significant theoretical and practical advances, two fundamental challenges remain open: 
\begin{enumerate}
    \item determining if a given controlled nonlinear system admits an infinite or finite-dimensional Koopman linear model; 
    \item synthesizing suitable lifting functions to obtain an accurate or approximate finite-dimensional Koopman linear model.
\end{enumerate}
The first challenge is particularly critical, as it reveals whether the modeling error arises from intrinsic system dynamics or from an inadequate choice of lifting functions. \textcolor{black}{Conditions for the existence of finite-dimensional Koopman linear embeddings (linear immersibility) have been established using function spaces generated by iterates of the system dynamics \cite{lee1988immersion,wang2021immersion}. However, these results do not explicitly characterize the underlying structure of the nonlinear system.} When an exact finite-dimensional Koopman linear model exists, data-driven approaches such as \cite[Theorem 3]{shang2024willems},\cite{shang2027regularization} can be leveraged to bypass the identification of explicit lifting functions. \textcolor{black}{Specifically, the state of the original nonlinear system can be treated as the output of the Koopman linear embedding. Under suitable data conditions, an extension of Willems' fundamental lemma \cite[Theorem 3]{shang2024willems} then provides a predictor directly from input-output data of the lifted system, without requiring access to the lifted state (\emph{i.e.}, lifting functions).}

\vspace{-3mm}
\subsection{Our contributions}
\vspace{-1mm}

In this work, we provide necessary and sufficient criteria for {\color{black}determining whether a controlled nonlinear system admits an exact Koopman linear model in both \textit{infinite-dimensional} and \textit{finite-dimensional} settings.} \textcolor{black}{Our results illustrate the existence of a Koopman linear embedding imposes a strong structural constraint on the nonlinear system and the existence of finite-dimensional embedding problem is reduced to that of an autonomous system.} The central insight is the existence of a Koopman linear embedding implies the system state must evolve linearly with respect to control input at all time steps. 

Building on this observation, we first introduce and study \textit{N-step input-affine predictors} for nonlinear systems. We show that $N$-step input-affine predictors with $N = 1,2, \ldots, \infty$ are equivalent to the existence of a control-affine preserved (\CAP{}) structure under an invertible coordinate transformation (\Cref{them:multi-steps}). This \CAP{} structure decomposes the dynamics into two interacting subsystems: an autonomous nonlinear subsystem and a control-driven subsystem that evolves linearly in its substate and input, with nonlinear terms induced solely by the autonomous component. {\color{black}This \CAP{} structure is indeed equivalent to the existence of a (potentially infinite-dimensional) Koopman linear embedding (\Cref{them:accur-Koop-function}).} We show that the existence of a finite-dimensional Koopman linear embedding further requires the autonomous nonlinear subsystem to admit its own exact Koopman linear representation, with a finite-dimensional invariant subspace whose span absorbs all nonlinearities of the full system (\Cref{them:accur-Koop}).

\textcolor{black}{Establishing these results requires several key proof techniques. In particular, for the general Koopman linear embedding (\Cref{them:multi-steps,them:accur-Koop-function}), we first link the Koopman linear embedding in the lifted space with the original nonlinear system by $N$-step input-affine predictability (\Cref{prop:Koopman-function-infty}). Then, to characterize nonlinear systems that admit $N$-step input-affine predictors, we~introduce a principle of affine composition that characterizes when affine dependence on the control input is preserved under recursive composition of the dynamics (\Cref{lemma:func-affine-rel}). Leveraging this principle, we construct a sequence of coordinate transformations which progressively concentrate all nonlinearities into an autonomous subsystem (\Cref{proposition:k+1-linear-predictor}), yielding the \CAP{} structure. To establish the finite-dimensional characterization in \Cref{them:accur-Koop}, we first show that a Koopman linear embedding can be chosen to be observable, with lifting functions that explicitly include the original state (\Cref{lem:koop-ob,lem:koop-seper}). By exploiting the equivalence between the multi-step state predictions generated by this embedding and those induced by the \CAP{} structure, we construct the required lifting functions for the autonomous subsystem from those of the controlled nonlinear system.}

Finally, we introduce a symbolic procedure in \Cref{alg:multi-check,alg:SCD} to verify whether a given system can be transformed into the \CAP{} structure. {\color{black} This identifies the existence of the Koopman linear embedding, and provides a principled way to assess the source of Koopman approximation errors. If a \CAP{} structure does not exist, the approximation error is intrinsic to the system dynamics. Accounting for modeling errors then becomes important. Robust control techniques, such as tube-based MPC \cite{zhang2022robust} and conservative surrogate constraints \cite{mamakoukas2022robust}, can be employed to mitigate the effects. Conversely, if a \CAP{} structure exists, selecting suitable lifting functions or increasing the lifting dimension may improve finite-horizon prediction accuracy. The nonlinear components in the \CAP{} structure may guide the selection of these lifting functions.}

\vspace{-2mm}
\subsection{Paper structure and notation}
The remainder of this paper is structured as follows. \Cref{sec:preliminary} provides a brief review of Koopman operator theory and our problem formulation. \Cref{sec:main-theorem} presents two main technical results on 1) the equivalence between \CAP{} structure and $N$-step input-affine predictability and 2) sufficient and necessary conditions for Koopman linear embeddings. \Cref{sec:accurate-multi-step} and \Cref{sec:accurate-Koopman} provide detailed proofs for our main theorems. A verification procedure for the \CAP{} structure is introduced in \Cref{sec:alg-cerification}. Finally, \Cref{sec:conclusion} concludes the paper. 

\textbf{Notation:} We represent the set of natural numbers~as~$\mathbb{N}$. For a series of vectors $a_1,\ldots, a_n$ and matrices $A_1, \ldots, A_n$, we denote $\col(a_1, \ldots, a_n) \!:=\! [a_1^\tr, \ldots, a_n^\tr]^\tr$ and $\col(A_1, \ldots, A_n) \!:=\! [A_1^\tr, \ldots, A_n^\tr]^\tr$. We represent the $2$-norm of a vector $a$ as $\|a\|_2$ and the Frobenius norm of a matrix $A$ as $\|A\|_F$. For a product of matrices, we use the convention $\prod_{i=n}^m A_i = A_{m} A_{m-1}\cdots A_n$ with $m \ge n \ge 0$. Suppose $x \in \mathbb{R}^n$ is the state of a dynamic system and we separate it into $k$ groups as $x := \col(x_1, \ldots, x_k)$ where $x_1 \in \mathbb{R}^{n_1}, \ldots, x_k \in \mathbb{R}^{n_k}$ and $n_1 + \cdots + n_k = n$. We represent $\col(x_i, x_{i+1}, \ldots, x_j)$ ($i \le j \le k$) as $x_{i:j}$ and denote $x(t) := \col(x_1(t), \ldots, x_k(t))$ as the system state at time step $t$. The term $x_{i:j}$ vanishes when $j < i$. We use $\mathbb{0}$ and $I_n$ to denote the zero matrix with compatible size and identity matrix with size $n$, respectively.

\section{Preliminaries and problem statement}
\label{sec:preliminary}
\subsection{Koopman operator for autonomous systems}
Consider a discrete-time autonomous dynamical system 
\begin{equation}
    \label{eqn:nonlinear-auto}
    x^+ = f(x),
\end{equation}
where $x \in \mathcal{X}\subseteq \mathbb{R}^n$ denotes the system state and $f: \mathcal{X} \to \mathcal{X}$ is the function describing the system~dynamics. Let $\mathcal{H}$ be a function space consisting of functions (called \emph{observables}) $\psi:\mathcal{X} \rightarrow \mathbb{C}$. \textcolor{black}{We assume that $\mathcal{H}$ is closed under composition with the dynamics $f$, that is, $\psi \circ f \in \mathcal{H}$ for all $\psi \in \mathcal{H}$ \cite{mezic2021koopman,korda2018linear,haseli2023modeling}.} The Koopman operator $\mathcal{K}:\mathcal{H} \rightarrow \mathcal{H}$ acts on the observables~via 
\begin{equation} \label{eq:Koopman-operator}
\mathcal{K} \psi(x) = \psi(f(x)).
\end{equation}
In other words, the Koopman operator evaluates an observable $\psi$ along the forward trajectories of the nonlinear system \cref{eqn:nonlinear-auto}. 

By construction, the Koopman operator \cref{eq:Koopman-operator} is linear, \emph{i.e.}, for any $\psi_1, \psi_2 \in \mathcal{H}$ and  $\alpha, \beta \in \mathbb{C}$, 
\begin{equation} \label{eq:linear-Koopman}
\mathcal{K}(\alpha \psi_1 + \beta \psi_2) = \alpha \mathcal{K} \psi_1 +\beta \mathcal{K} \psi_2.  
\end{equation}
Since $\mathcal{K}$ is a linear operator, we can study its spectral properties. We say the function $\psi \in \mathcal{H}$ is an eigenfunction of $\mathcal{K}$ with eigenvalue $\lambda \in \mathbb{C}$ if
$
\mathcal{K}\psi = \lambda \psi. 
$ 
By definition, a Koopman eigenfunction $\psi$ evolves linearly on any system trajectory, \emph{i.e.}, 
$$
\psi(x(t)) = \lambda \psi(x(t-1)) = \lambda^t \psi(x(0)), \quad \forall t \in \mathbb{N}, x(0) \in \mathcal{X}. 
$$

We next define the concept of Koopman invariant subspace. A subspace $\mathcal{S} \subseteq \mathcal{H}$ is \textit{Koopman invariant} if  $\mathcal{K} \psi \in \mathcal{S}$ for all $\psi \in \mathcal{S}$. When such a subspace is finite-dimensional, the action of the Koopman operator can be represented by a matrix once a basis is chosen. Specifically, let $\Psi:= \col(\psi_1, \ldots, \psi_{n_s})$ mapping from $\mathcal{X}$ to $\mathbb{C}^{n_s}$ be a basis of $\mathcal{S}$, where $n_s$ is its dimension. Then, there exists a matrix $A \in \mathbb{C}^{n_s\times n_s}$ such that 
\[
\mathcal{K} \Psi(x) = \Psi(f(x)) = A \Psi(x).
\]
For any function $\psi \in \mathcal{S}$, we can write $\psi = w^\tr \Psi$ for some $w \in \mathbb{C}^{n_s}$ and we have 
$$
\mathcal{K} \psi(x) = w^\tr \mathcal{K}\Psi(x) = w^\tr A \Psi(x),
$$
showing that every observable in $\mathcal{S}$ evolves linearly along system trajectories. Let $z := \Psi(x) \in \mathbb{C}^{n_s}$ denote the lifted state. We have $z^+ = \Psi(x^+) = \mathcal{K} \Psi(x) = A \Psi(x) = Az $. {\color{black}To analyze the original nonlinear dynamics using techniques from linear system theory via the Koopman operator, we require the state $x$ of the nonlinear system \eqref{eqn:nonlinear-auto} to be included in the space~$\mathcal{S}$. That leads to a Koopman linear representation 
\begin{equation}
\label{eqn:Koop-linear-auto}
z^+ = Az, \quad  x = Cz.
\end{equation}}

\vspace{-2mm}
When the finite-dimensional subspace $\mathcal{S}$ is not invariant, we can approximate the operator's action via projection operators. Let $P_\mathcal{S}$ be a projection operator on $\mathcal{S}$. Instead of $\mathcal{K}$, one can use the approximation operator $P_\mathcal{S} \mathcal{K}$ which admits $\mathcal{S}$ as an invariant subspace. A popular data-driven version of this projection-based approximation is the EDMD method~\cite{williams2015data}. Given a state sequence $x_0, \ldots, x_{n_\D-1}$, we organize it with the associated lifted state for a pre-selected basis $\Psi$ of $\mathcal{S}$ as 
\begin{equation}
    \label{eqn:EDMD-auto-traj}
    \begin{aligned}
    X \!=\! \begin{bmatrix}
        x_0, \ldots, x_{n_\D-2}
    \end{bmatrix}, & \quad Z \!=\! \begin{bmatrix}
        \Psi(x_0),\ldots,\Psi(x_{n_\D-2})
    \end{bmatrix}, \\
     X^+ \!=\! \begin{bmatrix}
        x_1, \ldots, x_{n_\D-1}
    \end{bmatrix}, & \quad 
     Z^+ \!=\! \begin{bmatrix}
        \Psi(x_1), \ldots, \Psi(x_{n_\D-1})
    \end{bmatrix}.
    \end{aligned}
\end{equation}
The EDMD estimate of the Koopman matrix is obtained by solving the following least-squares problems 
\begin{equation}
\label{eqn:EDMD-auto-least}
A \in \argmin_{A} \ \|Z^+ - AZ\|_F^2, \ \ C \in \argmin_{C} \ \|X - CZ\|_F^2.
\end{equation} 

It is known that autonomous nonlinear systems always admit an \textit{infinite-dimensional} Koopman linear representation \cite{mauroy2020koopman}. \textcolor{black}{Given a probability measure $\mu$ on $\mathcal{X}$, suppose $(f,\mathcal{X},\mu)$ is ergodic
and that the initial condition of \eqref{eqn:EDMD-auto-traj} is drawn from~$\mu$.} Under other suitable assumptions, as the dimension of the chosen subspace $\mathcal{S}$ (\emph{i.e.}, number of lifting functions) and the amount of data $n_\D$ increase, the \emph{finite-horizon} prediction error of the EDMD-based Koopman linear representation asymptotically converges to zero; see \cite[Theorem 5 and Remark 2]{korda2018convergence}. However, we note that this convergence is not monotonic~\cite{MH-JC:25-access}.

\subsection{Koopman Control Family and Koopman linear embedding} 
The Koopman operator framework has been extended to controlled systems in several different ways \cite{korda2018linear,haseli2023modeling}, which have recently been shown to be equivalent in \cite{MH-IM-JC:25-tac}. We here review the notion of Koopman control family from~\cite{haseli2023modeling}. Consider a discrete-time controlled nonlinear system
\begin{equation}
\label{eqn:nonlinear}
x^+ = f(x, u),
\end{equation}
where $x \in \mathcal{X} \subseteq \mathbb{R}^n$ and $u \in \mathcal{U} \subseteq \mathbb{R}^m$ denote the state and input of the system, and $f:\mathcal{X} \times \mathcal{U} \rightarrow \mathcal{X}$ is the function describing the system behavior. Fixing the input at constant values $u^*$ gives us a collection of autonomous systems 
\[
x^+ = f_{u^*}(x) := f(x, u^*), \quad u^* \in \mathcal{U}.
\]
The Koopman Control Family~\cite{haseli2023modeling} is defined as the family of operators $\{\mathcal{K}_{u^*}: \mathcal{H} \rightarrow \mathcal{H}\}_{u^* \in \mathcal{U}}$, where the function space $\mathcal{H}$ is assumed to be closed under composition with $f_{u^*}$ for all $u^* \in \mathcal{U}$. Each operator $\mathcal{K}_{u^*}$ acts on an observable $\psi \in \mathcal{H}$ as 
\begin{equation} \label{eq:Koopman-control-family}
\mathcal{K}_{u^*} \psi(x) = \psi(f_{u^*}(x)) = \psi(f(x, u^*)).
\end{equation}

The nonlinear system \eqref{eqn:nonlinear} admits an input-state separable form when a finite-dimensional common invariant subspace $\mathcal{S} \subseteq \mathcal{H}$ exists. The subspace $\mathcal{S} \subseteq \mathcal{H}$ is a \textit{common invariant subspace} if $\mathcal{K}_{u^*} \psi \in \mathcal{S}$ for all $u^* \in \mathcal{U}$ and $\psi \in \mathcal{S}$. Similar to the autonomous case, let $\Psi$ be a basis of $\mathcal{S}$ with dimension~$n_s$ and $z:=\Psi(x)$ be the lifted state. Then, there exists a function $\mathcal{A}:\mathcal{U}\! \rightarrow\! \mathbb{C}^{n_s \times n_s}$ such that, for all $u^* \! \in\!  \mathcal{U}$ (\!\cite[Theorem 4.3]{haseli2023modeling}),
\[
\mathcal{K}_{u^*} \Psi(x) = \Psi(f_{u^*}(x)) = \mathcal{A}(u^*) \Psi(x). 
\]
Thus, if the Koopman control family \cref{eq:Koopman-control-family} admits a common invariant subspace $\mathcal{S}$ with basis $\Psi$ and the state of nonlinear system \cref{eqn:nonlinear} is contained in $\mathcal{S}$, it admits the input-state separable form below \cite[Theorem 4.3]{haseli2023modeling}
\begin{equation}
\label{eqn:input-sep}
z^+ = \mathcal{A}(u)z, \quad x = Cz.
\end{equation}

The input-state separable form~\cref{eqn:input-sep} is much more structured compared with the original nonlinear form \cref{eqn:nonlinear}. However,~unlike the autonomous case \cref{eqn:Koop-linear-auto}, it is still not ready to use linear system tools to analyze or design controllers for \cref{eqn:input-sep} due to the coupling between the input $u$ and lifted state $z$. In practice, the \textit{Koopman linear embedding} has been widely used \cite{korda2018linear,haggerty2023control, mamakoukas2021derivative,shang2024willems,Shang2025KoopmanMPC,shang2025exponential}. This can be viewed as a special input-state separable form in which the lifting function evolves linearly with respect to both the lifted state $z$ and input $u$. We give a definition by adapting \cite[Definition 1]{shang2024willems}, \cite[Definition 1]{heeg2026limitations}.
\vspace{-1mm}
{\color{black}
\begin{definition}[Koopman linear embedding]  
\label{def:Koopman-linear}
The system \eqref{eqn:nonlinear} admits a Koopman linear embedding if there exists a real-valued function space\footnote{As the state $x$ and input $u$ are real-valued, we work with a real-valued lifting function space without loss of generality. Specifically, if the system admits a Koopman linear embedding with complex-valued lifting functions, there exists a corresponding real-valued one.}  $\mathcal{V}$ containing the coordinate functions
$\xi_i(x)=x_i$, $i=1,\ldots,n$, and there exist a reference input $\bar u\in\mathcal U$ and a linear map
$\mathcal B:\mathcal V\to\mathbb R^m$ 
such that, for all $\psi \in \mathcal{V}$ and $u^* \in \mathcal{U}$, we have $\mathcal{K}_{u^*} \psi \in \mathcal{V}$ and 
\begin{equation}
\label{eqn:Koopman-function-sp}
(\mathcal{K}_{u^*}\psi) (x) = (\mathcal{K}_{\bar{u}}\psi)(x) + \langle\mathcal{B}\psi, u^*-\bar{u}\rangle, \quad  \forall x \in \mathcal{X}.
\end{equation}
The embedding is finite-dimensional if
$\dim\!\mathcal{V}\!<\!\infty$. In this case, there exists $\Psi:=\mathrm{col}(\psi_1,\ldots,\psi_{n_z})$ (\emph{e.g.}, choosing the basis of $\mathcal{V}$;
see \Cref{apx:finite-dimensional-realization} for details), and constant matrices $A\in\mathbb R^{n_z\times n_z}$, $B\in\mathbb R^{n_z\times m}$, $C\in\mathbb R^{n\times n_z}$ such~that, for all $(x,u)\in\mathcal X\times\mathcal U$, 
\begin{equation}
\label{eqn:lifted-state}
\Psi(x^+) = \Psi(f(x,u)) = A\Psi(x) + Bu, \quad x = C\Psi(x).
\end{equation} 
\end{definition}}
\vspace{-1mm}

{\color{black}
Koopman linear embedding requires that the observables in $\mathcal V$ evolve under the reference Koopman operator $\mathcal K_{\bar u}$ (which is linear by definition), with an additive contribution that is linear in the input deviation $u^*-\bar u$ and independent of the state $x$. We require the state coordinates to lie in $\mathcal V$ as we focus on linear recovery of the state. Depending on the application, one may instead consider recovering outputs of the form $y=p(x)$ using lifted states or delayed input-output measurements \cite{korda2018linear,de2024koopman}. We leave the extension to these settings for future work. In the case of a finite-dimensional Koopman linear embedding, we refer to $z := \Psi(x) \in \mathbb{R}^{n_z}$ as the lifted state. From \cref{eqn:lifted-state}, the nonlinear system \cref{eqn:nonlinear} admits a Koopman linear model}
\begin{equation}
\label{eqn:Koopman-linear}
z^+ = Az + Bu, \quad x = Cz.
\end{equation}
Note that the Koopman linear model \eqref{eqn:Koop-linear-auto} can be viewed as a special case of the Koopman linear embedding~\cref{eqn:Koopman-linear}~with~zero input. Even when an exact Koopman linear embedding does not exist, approximate Koopman linear representations as \eqref{eqn:Koopman-linear} have been widely used in predictive control \cite{korda2018linear, haggerty2023control, mamakoukas2021derivative,Shang2025KoopmanMPC,shang2025exponential,shi2026koopman}.

Analogous to the autonomous case, we can apply the EDMD approach to approximate matrices \(A,\! B,\! C\) in~\eqref{eqn:Koopman-linear}. In addition to the state and lifted-state data in \eqref{eqn:EDMD-auto-traj}, we now also have the input sequence \(U = [u_0, \ldots, u_{n_\D-2}]\). The corresponding least-squares approximations take the form $(A, B)  \in \argmin_{A, B} \ \|Z^+ - AZ - BU \|_F^2, 
\ C  \in \argmin_{C} \ \|X - CZ\|_F^2.$
Unlike the autonomous setting, however, the prediction error of the Koopman linear model with control does not necessarily vanish as we increase the number of lifting functions and the amount of data. The main reason is that controlled nonlinear systems may fail to admit the exact Koopman linear embedding \cref{eqn:Koopman-function-sp} even in infinite dimensions~\cite{iacob2024koopman,haseli2023modeling}. 

\vspace{-2mm}
\subsection{Problem statement}
\label{subsec:prob-state}
A finite-dimensional Koopman linear embedding provides a powerful surrogate model for nonlinear systems. When such an embedding exists, the nonlinear dynamics can be represented exactly by a linear system in appropriately lifted coordinates. This representation enables the direct use of well-established tools from linear systems and control theory. This perspective has been widely adopted to model and control~nonlinear~dynamics, even in settings where the existence of an exact finite-dimensional embedding is not known \cite{korda2018linear, haggerty2023control, mamakoukas2021derivative, shang2024willems, Shang2025KoopmanMPC, shang2025exponential,shi2026koopman}.

A finite-dimensional Koopman linear embedding \cref{eqn:lifted-state} is a special case of the input-state separable form~\cref{eqn:input-sep}:
$$
\begin{bmatrix}
    \Psi(x^+)  \\ 1 
\end{bmatrix} = \begin{bmatrix}
    A & Bu \\ 0 & 1
\end{bmatrix}\begin{bmatrix}
    \Psi(x)  \\ 1 
\end{bmatrix}.
$$
Thus, the existence of a finite-dimensional Koopman linear embedding \cref{eqn:lifted-state} necessarily requires a common Koopman invariant subspace \cite[Theorem 4.3 and Lemma 4.6]{haseli2023modeling}. On the other hand, such embeddings~strictly extend beyond the classical classes of linear time-invariant (LTI) systems (where we can choose $\Psi(x) = x$) and~beyond the autonomous Koopman linear representation \cref{eqn:Koop-linear-auto}. They can represent a broad family of nonlinear systems whose dynamics become linear in suitable lifted coordinates, beyond what is achievable with classical linearization. This structural advantage makes Koopman linear embeddings especially attractive for control design, yet it also raises a fundamental question:
\vspace{-1mm}
\begin{center}
   {\color{black} \textit{Which class of nonlinear systems admits a Koopman linear embedding, and when is it finite-dimensional?}}
\end{center}
\vspace{-1mm}

Despite recent interest in data-driven and model-based Koopman methods \cite{korda2018linear, haggerty2023control, mamakoukas2021derivative, shang2024willems, proctor2018generalizing, Shang2025KoopmanMPC, shang2025exponential,shi2026koopman}, a precise characterization of this class of systems remains open. The main objective of this paper is to provide an exact, system-level characterization of all discrete-time controlled nonlinear systems that admit a Koopman linear embedding {\color{black}in both infinite-dimensional and finite-dimensional cases}. \textcolor{black}{Our results identify the necessary and sufficient structural conditions for the existence of exact Koopman linear embeddings. This characterization facilitates model development and error analysis for approximate Koopman linear models, which can remain effective in practice when modeling errors are appropriately accounted for, cf.~\cite{zhang2022robust,shang2025exponential}.} 

\vspace{-2mm}
\section{Koopman linear embedding and multi-step linear prediction}
\label{sec:main-theorem}
\vspace{-1mm}
We present the characterizations of Koopman linear embeddings. We begin with a motivating example that provides an informal preview of the underlying structure. We then~introduce a notion of multi-step input-affine predictors. \textcolor{black}{Using~this notion, we establish that the existence of a Koopman linear embedding is equivalent to the presence of a \textit{control-affine~preserved} (\CAP) structure, and the finite-dimensional embeddability additionally requires the autonomous subsystem to admit a suitable finite-dimensional Koopman invariant subspace.} 

\vspace{-2mm}
\subsection{Motivating example and informal characterizations} \label{subsection:example-Koopman}
\vspace{-1mm}
Consider the following two-dimensional nonlinear system, adapted from \cite{brunton2016koopman}: 
\begin{equation}
\label{eqn:slow-manifold}
\begin{aligned}
x_1^+  \!=\! x_2^2 + x_1 + u, \qquad 
x_2^+  \!=\! 0.9 x_2.
\end{aligned}
\end{equation}
This system is also known as a slow-manifold system and has been used as an academic example in the Koopman operator literature \cite{brunton2016koopman,shang2025exponential}. 
Let us consider the Koopman lifting  $\Psi(x) = \col(x_1, x_2, x_2^2)$.  It is straightforward to verify that this lifting evolves linearly along trajectories of \eqref{eqn:slow-manifold}, \emph{i.e.}, \\
\vspace{-1mm}
$$
\Psi(x^+) = \begin{bmatrix}
        1 & 0 & 1 \\ 0 &  0.9 & 0 \\ 0 & 0 & 0.81
    \end{bmatrix}\Psi(x) + \begin{bmatrix}
        1 \\ 0 \\ 0
    \end{bmatrix} u. 
\vspace{-1mm}
$$
By choosing the lifted state $z:=\Psi(x) \in \mathbb{R}^3$, we obtain a finite-dimensional Koopman linear model  
\begin{equation}
\label{eqn:slow-manifold-linear}
\begin{aligned}
    z^+ \!=\! \begin{bmatrix}
        1 & 0 & 1 \\ 0 &  0.9 & 0 \\ 0 & 0 & 0.81
    \end{bmatrix} z \! +\! \begin{bmatrix}
        1 \\ 0 \\ 0
    \end{bmatrix} u, \quad x = \begin{bmatrix}
        1 & 0 & 0 \\ 0 & 1 & 0
    \end{bmatrix} z. 
\end{aligned}
\end{equation}

Therefore, the slow-manifold system \eqref{eqn:slow-manifold} admits the finite-dimensional Koopman linear embedding \cref{eqn:slow-manifold-linear} per \Cref{def:Koopman-linear}. All nonlinear trajectories with arbitrary input from \eqref{eqn:slow-manifold} are embedded in those of the LTI system \cref{eqn:slow-manifold-linear}. \Cref{fig:Koopman-linear-embedding} illustrates three trajectories from the Koopman linear embedding \cref{eqn:slow-manifold-linear} and the original nonlinear dynamics \eqref{eqn:slow-manifold}. With an exact lifting of the initial condition, $z_0 = \Psi(x_0)$, the behaviors of \eqref{eqn:slow-manifold} are exactly embedded in those of \cref{eqn:slow-manifold-linear}, as shown by the first and second trajectories in \Cref{fig:Koopman-linear-embedding}. In contrast, when the initial lifting is inexact, the Koopman linear system may exhibit trajectories that do not correspond to any trajectory of the original nonlinear system, as shown by the third trajectory.

\begin{figure}[t]
    \centering
    \includegraphics[width=1\linewidth]{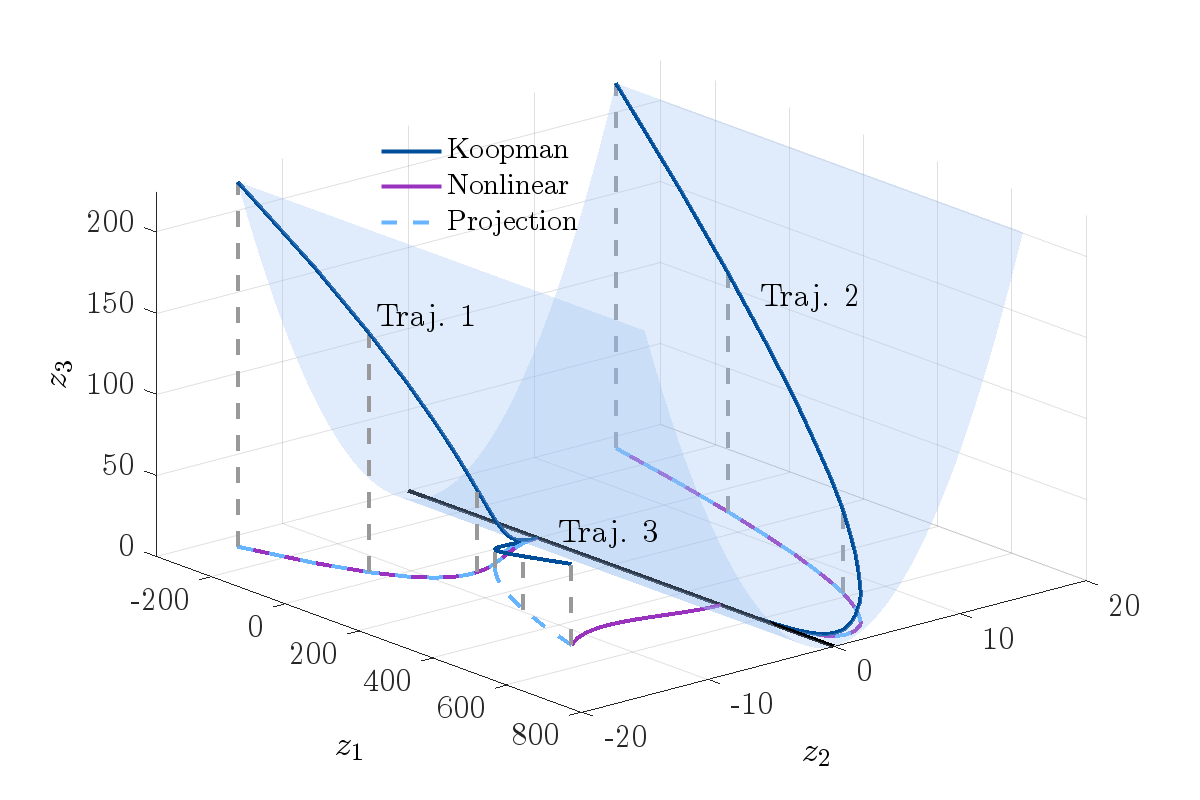}\\
    \vspace{-2mm}
    \caption{Visualization of the Koopman linear embedding for the slow-manifold system \cref{eqn:slow-manifold}. The blue surface represents the invariant manifold $z_3 = z_2^2$. The dark blue, light blue, and purple curves denote trajectories of the Koopman linear embedding \cref{eqn:slow-manifold-linear}, their projections on the original state space $x = Cz \in \mathbb{R}^2$, and the trajectories from the original nonlinear system \cref{eqn:slow-manifold}, respectively. The first two trajectories start from exact lifting of the initial conditions, $z_0 = \Psi(x_0)$, whereas the last trajectory originates from an inexact lifted initial condition. }
    \label{fig:Koopman-linear-embedding}
    \vspace{-3mm}
\end{figure}

While \eqref{eqn:slow-manifold} is primarily an academic example, it provides insights into the structural conditions for Koopman linear embedding. 
In particular, the system has the following properties: 
\begin{enumerate}
    \item The state $x_1$ affected by input $u$ propagates linearly with respect to itself and the input $u$, while the nonlinearity purely comes from the state $x_2$;
    \item The propagation of $x_2$ only depends on itself, which forms an autonomous subsystem;
    \item The autonomous subsystem of $x_2$ admits a Koopman linear representation by choosing $\bar{z}:=\bar{\Psi}(x_2) = \col(x_2, x_2^2)$, whose Koopman invariant subspace includes $x_2$ itself and the nonlinear term $x_2^2$ in the evolution of $x_1$.
\end{enumerate}

These properties are sufficient to construct a Koopman linear embedding for \Cref{eqn:slow-manifold}. In the rest of this section, we show that this structure is not coincidental, but in fact necessary for the existence of any finite-dimensional Koopman linear embedding. Specifically, we identify three structural properties that characterize this class of systems:  
\begin{enumerate}
    \item \textbf{Input-affine cascaded structure:} The state directly influenced by the input evolves affinely with respect to itself and the input, while all nonlinearities enter in auxiliary states;
    \item \textcolor{black}{\textbf{Autonomous nonlinear subsystem:}} The auxiliary states evolve autonomously, independent of both the control input and the states influenced by the input;
    \item \textcolor{black}{\textbf{Finite-dimensional Koopman invariant subspace:}} The autonomous subsystem admits a finite-dimensional Koopman invariant subspace that contains its own state and all nonlinear terms appearing in the controlled dynamics.
\end{enumerate}

Our \Cref{them:accur-Koop} below establishes that these three properties are both necessary and sufficient for a controlled nonlinear system to admit a finite-dimensional Koopman linear embedding.

\vspace{-2mm}

\subsection{Input-affine predictors and the \CAP{} structure} 
We now introduce the notion of multi-step input-affine predictors, which express the system's finite-horizon evolution as the sum of a nonlinear function of the initial state and~a~linear function of the input sequence. This provides a key link between nonlinear dynamics and Koopman linear embeddings.

{\color{black}
\begin{definition} [Input-affine predictor]
\label{def:multi-step}
    Consider system \eqref{eqn:nonlinear}. Let $x(N)= f(\cdots f(f(x,u_0),u_1)
        \cdots,u_{N-1})$ denote its state at step $N$, starting from the initial state $x$ under the input sequence $u_{0:N-1}$. 
    \begin{enumerate}
        \item The system admits an $N$-step input-affine predictor if there exist a function $\Phi_N: \mathcal{X} \rightarrow \mathbb{R}^n$ and a matrix $B_N \in \mathbb{R}^{n \times Nm}$ such that 
        $$
        x(N) = \Phi_N(x) + B_N u_{0:N-1}, \quad \forall x\in \mathcal{X}, u_{0:N-1} \in \mathcal{U}^N,
        $$
        where $\mathcal{U}^N := \mathcal{U} \times \cdots \times \mathcal{U}$ ($N$-times Cartesian product).
        \item The system admits $\infty$-step input-affine prediction if it admits an $N$-step input-affine predictor for every $N \in \mathbb{N}$. 
\end{enumerate}
\end{definition}}

{\color{black}It is clear that a finite-dimensional Koopman linear embedding \cref{eqn:lifted-state} implies $\infty$-step input-affine prediction, since the lifted dynamics allow exact linear propagation over arbitrary horizons. The following result shows that this implication also holds for infinite-dimensional Koopman linear embeddings satisfying \cref{eqn:Koopman-function-sp}, cf. \Cref{def:Koopman-linear}. 
\begin{proposition} 
    \label{prop:Koopman-function-infty}
    If system \eqref{eqn:nonlinear} admits a (potentially infinite-dimensional) Koopman linear embedding, then it admits $\infty$-step input-affine prediction.
\end{proposition}
\begin{proof}
Let $\bar u$ and $\mathcal B$ be as in \Cref{def:Koopman-linear}, and consider an arbitrary initial state $x\in\mathcal X$ and input sequence $u_{0:N-1}\in\mathcal U^N$.
Since $\mathcal K_{\bar u}\mathcal V\subseteq\mathcal V$, iterating \eqref{eqn:Koopman-function-sp} along the trajectory gives
\begin{equation} \label{eq:forward-propogation}
\psi(x(N))=(\mathcal K_{\bar u}^{N}\psi)(x)
\,+\,\sum_{k=0}^{N-1}
\left\langle
\mathcal B\mathcal K_{\bar u}^{N-1-k}\psi,\,
u_k-\bar u
\right\rangle,
\end{equation}
for every $\psi\in\mathcal V$. Here, $\mathcal K_{\bar u}^{j}$ denotes the $j$-fold composition of $\mathcal K_{\bar u}$, with $\mathcal K_{\bar u}^{0}=I$.
Thus, $(\mathcal K_{\bar u}^{N}\psi)(x)$ is the value of $\psi$ after $N$ steps starting from $x$ under the constant input $\bar u$. 

For $j=0,\ldots,N-1$, define the constant matrices
\[
M_j:=
\operatorname{col}\!\left(
(\mathcal B\mathcal K_{\bar u}^{j}\xi_1)^\top,\ldots,
(\mathcal B\mathcal K_{\bar u}^{j}\xi_n)^\top
\right)\in\mathbb R^{n\times m},
\]
where $\xi_i(x)=x_i$.
Because $\xi_1,\ldots,\xi_n\in\mathcal V$, applying \cref{eq:forward-propogation} to these coordinate functions yields
\[
x(N)=\Phi_N(x)+B_Nu_{0:N-1},
\]
where $B_N:=\begin{bmatrix}M_{N-1}&\cdots&M_0\end{bmatrix}$ and 
\[
\begin{aligned}
\Phi_N(x)
&:=\operatorname{col}\!\left(
(\mathcal K_{\bar u}^{N}\xi_1)(x),\ldots,
(\mathcal K_{\bar u}^{N}\xi_n)(x)
\right)
-\sum_{j=0}^{N-1}M_j\bar u.
\end{aligned}
\]
Note that $B_N$ is independent of the initial state and input sequence.
The system admits an $N$-step input-affine~predictor for any $N$, and thus $\infty$-step input-affine prediction holds.
\end{proof}
}

\Cref{prop:Koopman-function-infty} shows that $\infty$-step input-affine prediction is necessary for a Koopman linear embedding. This observation motivates us to decompose the characterization of Koopman linear embeddings into two subproblems: 
\begin{itemize}
    \item $\mathcal{P}_1:$ Develop sufficient and necessary conditions for \eqref{eqn:nonlinear} to admit  $\infty$-step input-affine prediction (cf. \Cref{them:multi-steps});
    {\color{black}\item $\mathcal{P}_2$: Identify the additional conditions beyond the $\infty$-step input-affine prediction for the existence of Koopman linear embeddings, both in general and in finite-dimensional function spaces (cf. \Cref{them:accur-Koop-function,them:accur-Koop}).}
\end{itemize}

\begin{remark} 
\label{rem:one-multi}
    When system \eqref{eqn:nonlinear}  admits a one-step input-affine predictor, it does not naturally guarantee the $\infty$-step input-affine prediction. This is because the affine dependency in the input may not be preserved under system recursion. For example, consider the scalar system 
    $x^+ = x^2 + u,$  
    which obviously has a one-step linear predictor, as $x(1) = x^+$. However, its 2-step evolution is given by  
    $$x(2) = x^4+2x^2u_0+u_0^2+u_1,$$ 
    that contains a nonlinear term $u_0^2$. Thus, additional structures are required to ensure that linear dependence on the input is preserved over multiple steps. \textcolor{black}{Similarly, the existence of an $N$-step input-affine predictor does not necessarily imply the existence of $k$-step input-affine predictors for $k<N$. This implication holds when the dynamics in \eqref{eqn:nonlinear} are surjective. Due to the page limit, we provide details in \Cref{apx:N-step-lower}.} \hfill $\circ$ 
\end{remark}

As illustrated in \Cref{rem:one-multi}, systems of the form $x^+ = \Phi(x) + Bu$ might not give rise to $\infty$-step input-affine prediction. To derive conditions for the $\infty$-step input-affine prediction problem~$\mathcal{P}_1$, our key insight is to ensure that affine dependence on the input $u$ is preserved during the system recursion. {\color{black}We will show this actually requires the system to be transformable to the control-affine preserved structure (\CAP{}).
\begin{definition}[Control-affine preserved structure]
\label{def:CAP}
Consider a nonlinear system $\tilde{x}^+ = \tilde{f}(\tilde{x}, u)$ with $\tilde{x}:=\col(\tilde{x}_1,\tilde{x}_2) \in\tilde{\mathcal{X}}\subseteq\mathbb{R}^n$ and $u\in\mathcal{U}\subseteq\mathbb{R}^m$,
where $\tilde{x}_1\in\mathbb{R}^{n_1}$, $\tilde{x}_2 \in \mathbb{R}^{n_2}$, $n_1,n_2\geq 0$, $n_1+n_2=n$ and $\tilde{f}:\tilde{\mathcal{X}} \times \mathcal{U} \rightarrow \tilde{\mathcal{X}}$. The system is in the \emph{control-affine preserved} (\CAP{}) form
if it can be written~as
\begin{equation}
\label{eq:CAP-structure}
    \begin{bmatrix}
        \tilde{x}_1^+\\
        \tilde{x}_2^+
    \end{bmatrix}
    =
    \begin{bmatrix}
        g(\tilde{x}_2)+C\tilde{x}_1\\
        h(\tilde{x}_2)
    \end{bmatrix}
    +
    \begin{bmatrix}
        D\\
        0
    \end{bmatrix}u,
\end{equation}
where $g: \mathrm{proj}_{\tilde{x}_2}(\tilde{\mathcal{X}})\rightarrow \mathbb{R}^{n_1}$, $h : \mathrm{proj}_{\tilde{x}_2}(\tilde{\mathcal{X}})\rightarrow\mathrm{proj}_{\tilde{x}_2}(\tilde{\mathcal{X}})$ and $C\in\mathbb{R}^{n_1\times n_1}$, $D\in\mathbb{R}^{n_1\times m}$ are constant matrices.
\end{definition}}
The control-affine preserved (\CAP) structure in \cref{eq:CAP-structure} formalizes the \textit{input-affine cascaded structure} and the presence of an \textit{autonomous nonlinear generator} discussed in \cref{subsection:example-Koopman}. This \CAP{} structure ensures that the affine dependence on the control input  $u$ is preserved during the system propagation over arbitrary horizons. \textcolor{black}{A specific form of the \CAP{} structure with a polynomial function $g$ and a linear mapping  $h$ is shown in \cite{wang2021immersion} as a sufficient condition for the existence of a finite-dimensional Koopman linear embedding.}

We now present our first main technical result, which establishes the sufficiency and necessity of \CAP{} structure for the $\infty$-step input-affine prediction.
\begin{theorem}[Equivalence of $\infty$-step input-affine prediction and \CAP{} form]
\label{them:multi-steps}
Consider system~\cref{eqn:nonlinear}, with $\mathcal{X}$ open and convex, $\mathcal{U}$ open, and $f\!: \!\mathcal{X} \times\mathcal{U} \! \rightarrow \! \mathcal{X}$~surjective. \textcolor{black}{Then, it admits $\infty$-step input-affine prediction if and only if there exists an invertible matrix $T \in \mathbb{R}^{n \times n}$ such that,} {\color{black}with~the coordinate transformation $\tilde{x}:= \col(\tilde{x}_1, \tilde{x}_2):= T x \in \tilde{\mathcal{X}}$, $\tilde{\mathcal{X}} := \{T x \in \mathbb{R}^n \mid x \in \mathcal{X}\}$}, the dynamics $\tilde{x}^+ = Tf(T^{-1}\tilde{x},u)$ takes the \CAP{} form in \eqref{eq:CAP-structure}.
\end{theorem}

The sufficiency of the \CAP{} structure in \cref{eq:CAP-structure} is immediate, since the input $u$ never enters the nonlinear component $\tilde{x}_2$; consequently, the system admits $N$-step input-affine predictors for any $N \in \mathbb{N}$. In contrast, establishing the necessity of the \CAP{} structure is nontrivial. \textcolor{black}{Our proof uses input-affine predictors of horizons $N\in\{1,\ldots,n+1\}$ and leverages structural properties of compositions of affine functions over open and convex sets.} The detailed arguments are technically involved, and we provide them in \Cref{sec:accurate-multi-step}.

\vspace{-3mm}

\subsection{Koopman linear embeddings}

{\color{black} \Cref{prop:Koopman-function-infty} and \Cref{them:multi-steps} show that a Koopman linear embedding requires the system to admit a \CAP{} structure up to an invertible linear coordinate transformation. 
 We next establish that the converse is also true. 
\begin{theorem}[Characterization of Koopman linear embeddings]
\label{them:accur-Koop-function}
    Consider system \eqref{eqn:nonlinear}, where $\mathcal{X}$ is open and convex, $\mathcal{U}$ is open, and $f:\mathcal{X}\times\mathcal{U}\rightarrow\mathcal{X}$ is surjective.
The system admits a Koopman linear embedding if and only if it can be transformed into the \CAP{} form \eqref{eq:CAP-structure} by an invertible linear coordinate transformation.
\end{theorem}
The ``only if'' direction follows directly from \Cref{prop:Koopman-function-infty,them:multi-steps}.
The ``if'' direction is based on construction: given the \CAP{} structure, we construct a function space $\mathcal{V}$ and a linear operator $\mathcal{B}$ satisfying \eqref{eqn:Koopman-function-sp}; details are provided in \Cref{subsection:koop-function-sp}. Thus, \Cref{them:accur-Koop-function} fully characterizes Koopman linear embeddability through the \CAP{} structure, reducing the existence question to identifying this structure under an invertible linear coordinate transformation. In \Cref{sec:alg-cerification}, we develop a tractable algorithm to certify whether a controlled nonlinear system admits such a transformation.

\begin{remark}
    We note that a nonlinear system that does not admit a Koopman linear embedding cannot be approximated arbitrarily accurately by
    Koopman linear models. Suppose the arbitrarily accurate Koopman linear approximations exist. Following a similar argument to \Cref{prop:Koopman-function-infty}, the system can be approximated arbitrarily well by an $N$-step input-affine predictor for $N\ge 1$. Since the class of affine functions is closed under uniform limits, the system must admit an $\infty$-step input-affine prediction (see detailed proof in \Cref{appendix:linear-approximation}). By \Cref{them:multi-steps,them:accur-Koop-function}, this implies the existence of a Koopman linear embedding, yielding a contradiction.
\end{remark}

The existence of a \CAP{} structure alone does not guarantee a finite-dimensional Koopman linear embedding. This additionally requires the autonomous nonlinear subsystem to admit a suitable finite-dimensional Koopman invariant subspace. The following result formalizes this characterization.} 

{\color{black}\begin{theorem}[Characterization of finite-dimensional Koopman linear embeddings]
\label{them:accur-Koop}
Consider system \eqref{eqn:nonlinear}, where $\mathcal{X}$ is open and convex, $\mathcal{U}$ is open, and $f:\mathcal{X}\times\mathcal{U}\rightarrow\mathcal{X}$ is surjective. 
 Then, the following two statements are equivalent:
    \begin{enumerate}
        \item \textbf{Finite-dimensional Koopman linear embedding:} there exist a lifting function $\Psi: \mathcal{X} \to \mathbb{R}^{n_z}$ and  constant matrices $A_\K, B_\K, C_\K$ with compatible dimensions, such that
        \begin{equation}
        \label{eqn:Koop-thm2}
       \Psi(f(x, u)) = A_\K \Psi(x) + B_\K u, \quad x =C_\K \Psi(x). 
        \end{equation}
    \item \textbf{\CAP{} structure with Koopman-closed autonomous subsystem:} Up to an invertible transformation $T \in \mathbb{R}^{n \times n}$, the system dynamics $\tilde{x}^+ = Tf(T^{-1}\tilde{x}, u)$ in the new coordinate $\tilde{x}\!:=\! \col(\tilde{x}_1, \tilde{x}_2)\! :=\! Tx \in \tilde{\mathcal{X}}$ with $\tilde{\mathcal{X}} := \{T x \in \mathbb{R}^n \mid x \in \mathcal{X}\}$ is in the \CAP{} form \eqref{eq:CAP-structure}. Moreover, there exist a lifting function \(\bar{\Psi}\!:\!\mathrm{proj}_{\tilde{x}_2}(\tilde{\mathcal{X}})\!\rightarrow\! \mathbb{R}^{\bar{n}_z}\)~and constant matrices \(\bar{A}_{\K}, \bar{C}_{\K}\) of compatible dimensions such that, for $g,h$ in \eqref{eq:CAP-structure},  
        \begin{equation} \label{eq:Koopman-closed-autonomous}
        \bar{\Psi}(h(\tilde{x}_2)) \!=\! \bar{A}_\K \bar{\Psi}(\tilde{x}_2), \; \col(\tilde{x}_2, g(\tilde{x}_2))\!=\!\bar{C}_\K \bar{\Psi}(\tilde{x}_2).
        \end{equation}
    \end{enumerate}
\end{theorem}}

Condition \cref{eq:Koopman-closed-autonomous} characterizes the \textit{finite-dimensional Koopman invariant space} property in \Cref{subsection:example-Koopman}. {\color{black}It gives a Koopman invariant subspace under the autonomous system $\tilde{x}_2^+ = h(\tilde{x}_2)$ that contains its own state $\tilde{x}_2$ and all nonlinear terms $g(\tilde{x}_2)$, and $h$ maps the set $\mathrm{proj}_{\tilde{x}_2}(\tilde{\mathcal{X}})$ to itself. Sufficient conditions for \eqref{eq:Koopman-closed-autonomous} have been established in Koopman theory for autonomous systems; see, \emph{e.g.}, \cite{wang2023computation}. Conditions in statement~2 present all required structural properties for a nonlinear system to admit a finite-dimensional Koopman linear embedding. It is clear that statement~2~implies statement~1. Combining \eqref{eq:CAP-structure} and \cref{eq:Koopman-closed-autonomous}, one can construct the lifting function as $\Psi(x):= \hat{\Psi}(Tx)$ with $\hat{\Psi}(\tilde{x}_1, \tilde{x}_2) := \col(\tilde{x}_1, \bar{\Psi}(\tilde{x}_2))$ and the associated matrices of the Koopman linear embedding become
\begin{equation} \label{eq:Koopman-linear-model-construction}
\begin{aligned}
A_\K = \begin{bmatrix}
    C & E  \\ \mathbb{0} & \bar{A}_\K
\end{bmatrix},  B_\K =  \begin{bmatrix}
    D \\ \mathbb{0}
\end{bmatrix}, C_\K = T^{-1}
\begin{bmatrix}
    I_{n_1} & \mathbb{0} \\ \mathbb{0} & F
\end{bmatrix},
\end{aligned}
\end{equation}
where $E:= [\mathbb{0} \ I_{n_1}] \bar{C}_\K, F:= [I_{n_2} \ \mathbb{0}] \bar{C}_\K$ and $n_\z = n_1 + \bar{n}_\z$.}

The proof that statement 1 implies statement 2 is technically involved. One key idea is to utilize the equivalence of state evolutions induced by the \CAP{} structure and the Koopman linear embedding. We show functions $g(\cdot)$ and $h(\cdot)$ must satisfy the property of the Koopman-closed generator in \cref{eq:Koopman-closed-autonomous} by enforcing that both formulations produce identical state trajectories. We provide proof details in \Cref{subsection:properties,subsection:proof-Koopman}. 

We have a few remarks on \Cref{them:accur-Koop}.
\vspace{-1mm}
\begin{remark}
The structural conditions in \Cref{them:accur-Koop} bridge two well-understood classes of systems. When the autonomous nonlinear subsystem is absent, the framework reduces to classical LTI systems. When the control input is removed, the conditions recover the standard setting of finite-dimensional Koopman linear representations for autonomous systems as in~\cref{eqn:Koop-linear-auto}. In this sense, the proposed structure provides a unifying perspective that connects linear control systems and Koopman operator theory for autonomous systems. \hfill $\circ$ 
\end{remark}

\vspace{-3mm}
\begin{remark}
    It is known that any affine system admits a Koopman linear embedding \cite{shang2024willems}. We show that the~class of affine systems is a special case of \Cref{eq:CAP-structure}, where the autonomous nonlinear generator vanishes. Consider 
    an affine system
    \vspace{-1mm}
    \begin{equation}
        \label{eqn:affine-sys}
        x^+ = Ax + Bu + r, 
        \vspace{-1mm}
    \end{equation}
    where $x \in \mathbb{R}^n$ is the state, $u \in \mathbb{R}^m$ is the input, and $r\in\mathbb{R}^n$ is a constant vector. In this case, both the function $g$ and the lifting function $\bar{\Psi}$ degrade to constants with $g(x_2) = r$ and $\bar{\Psi}(x_2) = 1$. Thus, we can construct the lifting function for~\eqref{eqn:affine-sys} as $\Psi(x) = \col(x, 1)$ and the associated matrices for the Koopman linear embedding are 
    \begin{align}
    A_\K = \begin{bmatrix}
        A & r \\ \mathbb{0} & 1
    \end{bmatrix}, \ B_\K = \begin{bmatrix}
        B \\ \mathbb{0}
    \end{bmatrix}, \ C_\K = \begin{bmatrix}
        I_n & \mathbb{0}
    \end{bmatrix}. \tag*{$\circ$}
    \end{align}
\end{remark}

\vspace{-2mm}

{\color{black}
\begin{remark}
\label{rem:nonlinear-recovery}
    In \Cref{them:accur-Koop}, we restrict the state recovery and the transformation to the \CAP{} structure to linear mappings. Allowing them to be nonlinear leads to an interesting open problem, that is: how to characterize the nonlinear systems which admit a Koopman linear embedding with nonlinear state recovery (\emph{i.e.}, $x = q(\Psi(x))$)? A sufficient condition follows by allowing the transformation and recovery in statement~2 of \Cref{them:accur-Koop} to be nonlinear. Another interesting question is under what conditions a nonlinear system admits a Koopman linear embedding with respect to a transformed input $v$ through $u=\alpha(x,v)$, where $\alpha$ is bijective for fixed $x$ \cite{fan2024learning}. We leave these characterizations for future work. \hfill $\circ$ 
    \end{remark}
}
We conclude this section with an example illustrating \Cref{them:accur-Koop}. In this example, the existence of a finite-dimensional Koopman linear embedding may not seem immediately apparent. Nevertheless, this embedding can be identified using a constructive procedure for $N$-step input-affine predictors. The details of this procedure are presented in \Cref{sec:alg-cerification}.
\vspace{-1mm}
\begin{example}
\label{example:concise}
Consider the nonlinear system 
\[
    \begin{bmatrix}
        x_1^+ \\ x_2^+ \\ x_3^+
    \end{bmatrix}\! =\! 
    \begin{bmatrix}
        (x_2\!+\!x_3)^2 \!+\! (x_1\!+\!x_2)\!+\!u\\
        (x_2\!+\!x_3)^2\!\cdot\!(x_2\!+\!x_3-1)\!+\!x_1 \!-\!2u\\
        (x_2\!+\!x_3)^2\!\cdot\!(1\!-\!x_2\!-\!x_3)\!-\!x_1\!+\!0.5x_2\!+\!0.5x_3\! +\! 2u
    \end{bmatrix},
\]
where $x := \col(x_1, x_2, x_3) \in \mathbb{R}^3$ is the state of the system. It admits an exact Koopman linear embedding with 
\[
z \!=\! \Psi(x) \!:=\! \col(x_1, -\!2x_1\!-\!x_2, x_2\!+\!x_3, (x_2\!+\!x_3)^2, (x_2\!+\!x_3)^3)
\]
and we can verify that 
\begin{equation} \label{eq:Koopman-model-example}
\begin{aligned}
z^+ =& \begin{bmatrix}
    -1 & -1 & 0 & 1 & 0 \\ 1 & 2 & 0 & -1 & -1\\ 0 & 0 & 0.5 & 0 & 0\\ 0 & 0 & 0 & 0.25 & 0 \\ 0 & 0 & 0 & 0 & 0.125
\end{bmatrix} z + \begin{bmatrix}
    1 \\ 0 \\ 0 \\ 0 \\ 0
\end{bmatrix} u, \\
x =& \begin{bmatrix}
    1 & 0 & 0 & 0 & 0 \\ 
    -2 & -1 & 0 & 0 & 0\\
    2 & 1 & 1 & 0 & 0
\end{bmatrix} z.
\end{aligned}
\end{equation}
This Koopman linear embedding \cref{eq:Koopman-model-example} may not seem immediate. 
Indeed, we can find an invertible transformation $T$ (see \Cref{sec:alg-cerification}) to transform the system into the \CAP{} structure. In the new coordinate $\tilde{x} := \col(\tilde{x}_1, \tilde{x}_2, \tilde{x}_3) = Tx$, we have 
\vspace{-1mm}
\begin{equation} \label{eq:CAP-example}
\begin{bmatrix}
    \tilde{x}_1^+ \\ \tilde{x}_2^+ \\ \tilde{x}_3^+
\end{bmatrix} = \begin{bmatrix}
    \tilde{x}_3^{2} - \tilde{x}_1- \tilde{x}_2 + u \\ -\tilde{x}_3^{3}-\tilde{x}_3^{2}+\tilde{x}_1+2\tilde{x}_2 \\ 0.5 \tilde{x}_3
\end{bmatrix}.
\vspace{-1mm}
\end{equation}
Then, we can design the associated lifting function for $\tilde{x}_3^{+} = h(\tilde{x}_3) := 0.5\tilde{x}_3$ as $\bar{\Psi} := \col(\tilde{x}_3, \tilde{x}_3^{2}, \tilde{x}_3^{3})$. The function $\bar{\Psi}$ includes all nonlinearities of the system (\emph{i.e.}, $\tilde{x}_3^{2}, \tilde{x}_3^{3}$) and evolves linearly as $h$ is a linear function. From \cref{eq:CAP-example}, it is now straightforward to obtain the Koopman linear model \cref{eq:Koopman-model-example}. \hfill $\circ$
\end{example}

\section{Existence of $N$-step input-affine predictors}
\label{sec:accurate-multi-step}

In this section, we discuss the existence of $N$-step input-affine predictors and provide the proof of \Cref{them:multi-steps}. According to \Cref{def:multi-step}, a {one-step input-affine predictor} already requires the system \cref{eqn:nonlinear} to be in the form of 
\begin{equation}
\label{eq:one-step-linear}
    f(x,u) = \Phi(x) + Bu, \quad \forall x \in \mathcal{X}, u \in \mathcal{U} ,
\end{equation}
where $\Phi$ is a nonlinear mapping and $B \in \mathbb{R}^{n \times m}$. The existence of $N$-step input-affine predictors with $N \in \{1, 2,\ldots, \bar{N}\}$ and $\bar{N} \ge 2$ must require the system \cref{eqn:nonlinear} to have structure beyond \cref{eq:one-step-linear}. We shall show that, with  $\bar{N} = n+1$, the system dynamics \cref{eqn:nonlinear} need to be in the \CAP{} structure, defined in \Cref{eq:CAP-structure}.

We assume the system \eqref{eq:one-step-linear} satisfies the assumption below.
\vspace{-3mm}

\begin{assumption}
\label{assum:one-step-predictor}
    The mapping $(x,u) \mapsto \Phi(x)+Bu$ from $\mathcal{X} \times \mathcal{U}$ to $\mathcal{X}$ is surjective, $\mathcal{X}$ is open and convex, and $\mathcal{U}$ is open.
\end{assumption}
{\color{black} The openness of $\mathcal{X}$ and $\mathcal{U}$ allows local perturbations of the state and input, and the convexity of $\mathcal{X}$ ensures the sets obtained by coordinate projection or by fixing some state coordinates of $\mathcal{X}$ remain convex and hence connected. We use these properties to develop conditions under which the composition of affine functions remains affine (see \Cref{lemma:func-affine-rel}). The system dynamics are forward invariant, and the surjectivity condition is imposed jointly on the state and input. They ensure that the domain of the function remains unchanged at each step of system propagation, allowing the structural conditions obtained from multi-step compositions to hold throughout $\mathcal{X}$. Note that~$\mathcal{X}$ and $\mathcal{U}$ can be chosen as domains on which the system dynamics are well defined and need not coincide with the state and input constraint sets for control design. Practical state and input constraints may be imposed on smaller and potentially closed subsets of $\mathcal{X}$ and $\mathcal{U}$.}

\vspace{-3mm}
\subsection{Affine functions on general sets and their composition}
One key element in our technical proof lies in the analysis of the composition of affine functions.
We begin by recalling the notion of affine functions~\cite{boyd2004convex}. 
\vspace{-1mm}
\begin{definition}[Affine function on a set]
A function $\Phi : \mathcal{X} \to \mathbb{R}^p$ is \emph{affine on $\mathcal{X} \subseteq \mathbb{R}^n$} if there exist a matrix $C \in \mathbb{R}^{p \times n}$ and a vector $v \in \mathbb{R}^p$ such that
$
    \Phi(x) = Cx + v, \, \forall x \in \mathcal{X}.
$ 
\vspace{-1mm}
\end{definition}

This definition does not require $\mathcal{X}$ to be convex or connected. An affine function is simply the restriction of a globally affine map to the set $\mathcal{X}$. We next state a basic result from \cite[Exercise 9.9]{rudin1976principles} that will be used in our key \cref{lemma:func-affine-rel} later. We also present a simple proof for completeness. 
\vspace{-1mm}
\begin{lemma}
\label{lem:constant-gradient-affine}
Let $\mathcal{X} \subseteq \mathbb{R}^n$ be open and connected, and let $\Phi : \mathcal{X} \to \mathbb{R}^p$ be a differentiable function. If there exists a constant matrix $C \in \mathbb{R}^{p \times n}$ such that
$
    \nabla \Phi(x) = C, \, \forall x \in \mathcal{X},
$
then $\Phi$ is affine on $\mathcal{X}$; that is, there exists $v \in \mathbb{R}^p$ such that
\[
    \Phi(x) = Cx + v, \qquad \forall x \in \mathcal{X}.
\vspace{-1mm}
\]
\end{lemma}
\begin{proof}
Define $F(x) := \Phi(x) - Cx$. Then $\nabla F(x) = 0$ for all $x \in \mathcal{X}$.  
Since $\mathcal{X}$ is open and connected, it is path-connected. For any $x_0, x \in \mathcal{X}$, let $\gamma : [0,1] \to \mathcal{X}$ be a continuously differentiable path with $\gamma(0) = x_0$ and $\gamma(1) = x$. By the chain rule, we have 
\[
    \frac{d}{dt} F(\gamma(t)) = \nabla F(\gamma(t)) \dot{\gamma}(t) = 0,
\]
which implies $F(\gamma(t))$ is constant along the path. Hence~$F(x)$ $= F(x_0), \forall x  \in  \mathcal{X}$, and the claim follows by $v :=F(x_0)$.
\end{proof}

The openness of $\mathcal{X}$ ensures that the Jacobian is well defined at all points, while connectedness guarantees that the affine offset is constant over $\mathcal{X}$. If $\mathcal{X}$ is disconnected, we may not have a common offset $v$ for the function $\Phi$. Affine functions enjoy simple and well-known closure properties under composition. Let $\Phi_1(x) = A_1 x + b_1$ and $\Phi_2(x) = A_2 x + b_2$ be affine functions with compatible dimensions. Then, the composition $\Phi_2 \circ \Phi_1$ is also affine, given~by
$
    (\Phi_2 \circ \Phi_1)(x) = A_2 A_1 x + (A_2 b_1 + b_2). 
$ 
This property is simple yet fundamental in linear systems theory. In the context of controlled nonlinear systems, however, affine dependence on the input $u$ may be destroyed under composition due to nonlinearities in the state (see \Cref{rem:one-multi}), unless additional constraints are imposed.

{\color{black}
We next formalize when the affine structure is necessarily preserved under the composition of affine functions. This result plays a crucial role in establishing the necessity of the \CAP{} structure for $N$-step input-affine prediction, as it allows us to successively derive structural constraints from input-affine predictors with horizons $N=2,\ldots,n+1$. 
\begin{lemma} 
\label{lemma:func-affine-rel}
    Let $\mathcal{X} \subseteq \mathbb{R}^n$ and $\mathcal{U} \subseteq \mathbb{R}^m$ be open, $g_1 : \mathcal{X} \to \mathbb{R}^{n_1}$, $l_1 : \mathcal{X} \to \mathbb{R}^{n_2}$ be arbitrary functions,  and $C_1 \in \mathbb{R}^{n_1 \times m}$ of full row rank. Consider  $\bar{\mathcal{X}} := \{(g_1(x)+C_1 u, l_1(x)) \in \mathbb{R}^{n_1 + n_2} \mid x \in \mathcal{X}, u \in \mathcal{U}\}$, and let $f:\bar{\mathcal{X}} \rightarrow \mathbb{R}^{n_3}$ be an arbitrary function. Suppose that:
    \begin{enumerate}[label=(\textbf{A\arabic*})]
    \item The set $\bar{\mathcal{X}}$ is open and convex;
    \item There exist a function $g_2:\mathcal{X}\to \mathbb{R}^{n_3}$ and a matrix $C_2\in\mathbb{R}^{n_3\times m}$ such that for all $x\in\mathcal{X}, u\in \mathcal{U}$, we have 
    \begin{equation}
    \label{eq:affine-recursion-identity-general}
        f(g_1(x) + C_1 u, l_1(x)) = g_2(x) + C_2u. 
    \end{equation}
\end{enumerate}
Then, $f$ is affine in \(x_1\), i.e., there exists a matrix $C \in \mathbb{R}^{n_3 \times n_1}$ and a function $g:\mathrm{proj}_{x_2}(\bar{\mathcal{X}}) \to \mathbb{R}^{n_3}$ such that
    \[
    f(x_1, x_2) = C x_1 + g(x_2), \quad \forall (x_1, x_2) \in \bar{\mathcal{X}}.
    \]
\end{lemma}

\begin{proof}
Since $C_1$ has full row rank, its pseudoinverse
$C_1^\dagger=C_1^\top(C_1C_1^\top)^{-1}$ satisfies
$C_1C_1^\dagger=I_{n_1}$. Define
$
    C:=C_2C_1^\dagger\in\mathbb{R}^{n_3\times n_1}.
$ 
Fix any $x_2\in\operatorname{proj}_{x_2}(\bar{\mathcal X})$ and consider the nonempty slice
$
    S_{x_2}:=\{x_1\in\mathbb{R}^{n_1}:
    (x_1,x_2)\in\bar{\mathcal X}\}.
$ 

For any $x_1\in S_{x_2}$, the definition of $\bar{\mathcal X}$ guarantees the existence of $\hat x\in\mathcal X$ and $\hat u\in\mathcal U$ such that
\begin{equation} \label{eq:def-x1-x2}
    x_1=g_1(\hat x)+C_1\hat u,
    \qquad x_2=l_1(\hat x).
\end{equation}
Because $\mathcal U$ is open, $\hat u+C_1^\dagger\delta\in\mathcal U$ for all sufficiently small $\delta\in\mathbb{R}^{n_1}$. For each such $\delta$, we have $(x_1+\delta,x_2)\in\bar{\mathcal X}$ from \cref{eq:def-x1-x2}. Then the assumption (A2) yields
\begin{equation} \label{eq:f-derivative}
\begin{aligned}
    f(x_1\!+\!\delta,x_2)
    &=f\bigl(g_1(\hat x)\!+\!C_1(\hat u\!+\!C_1^\dagger\delta),
             l_1(\hat x)\bigr)\\
    &=g_2(\hat x)\!+\!C_2(\hat u\!+\!C_1^\dagger\delta)=f(x_1,x_2)\!+\!C\delta.
\end{aligned}
\end{equation}
Define $F_{x_2}:S_{x_2}\to\mathbb{R}^{n_3}$ by
$F_{x_2}(x_1):=f(x_1,x_2)$. The identity \cref{eq:f-derivative} implies
\[
    \lim_{\delta\to 0}
    \frac{\|F_{x_2}(x_1+\delta)-F_{x_2}(x_1)-C\delta\|}
         {\|\delta\|}
    =0.
\]
Since $x_1\in S_{x_2}$ is arbitrary, $F_{x_2}$ is differentiable on $S_{x_2}$ with constant Jacobian
$
    \nabla F_{x_2}(x_1)=C,\, \forall x_1\in S_{x_2}.
$ 
By the assumption (A1), $S_{x_2}$ is open and convex, hence connected.
From \Cref{lem:constant-gradient-affine}, there exists a vector $g(x_2)\in\mathbb{R}^{n_3}$ such that
\[
    F_{x_2}(x_1)=Cx_1+g(x_2),
    \qquad \forall x_1\in S_{x_2}.
\]
Since $x_2$ is arbitrary, this defines a function
$g:\operatorname{proj}_{x_2}(\bar{\mathcal X})\to\mathbb{R}^{n_3}$
satisfying
$
    f(x_1,x_2)=Cx_1+g(x_2),
    \, \forall (x_1,x_2)\in\bar{\mathcal X},
$ 
which completes the proof.
\end{proof}

\Cref{lemma:func-affine-rel} shows that, under mild conditions, $f(g_1(x)+C_1u,l_1(x)) = g_2(x)+C_2u$ implies $f$ must be affine in its first argument. The key proof idea above is to fix $x_2$ and apply \Cref{lem:constant-gradient-affine} to the function $F_{x_2}(x_1) := f(x_1,x_2)$, whose Jacobian is constant. The full row rank condition on $C_1$ is essential, as it ensures perturbations in $u$ can induce variations in every direction of $x_1$. In deriving the \CAP{} structure, suitable state decompositions ensure that this condition holds for the relevant substate at each iteration.  Thus, applying \Cref{lemma:func-affine-rel} requires no additional rank assumption on the dynamics \eqref{eq:one-step-linear}. 
}

\vspace{-6mm}

{\color{black}
\subsection{Recursive strategy and state decomposition}\label{sec:low-dimensional}
Before presenting the proof of \Cref{them:multi-steps}, we illustrate the recursive idea by examining the additional structure imposed by a 2-step input-affine predictor. Specifically, we show that if \eqref{eq:one-step-linear} admits such a predictor, then either the system already has a \CAP{} structure or its nonlinearity can be concentrated into a lower-dimensional substate. In \Cref{subsec:general-nonlinear}, we extend this argument to multi-step input-affine prediction, showing that each recursive iteration either yields a \CAP{} structure or further reduces the dimension of the nonlinear substate. Repeating this procedure leads to the desired \CAP{} structure in \Cref{them:multi-steps}.

\textbf{Case I (\CAP{} structure):} In this case, we consider $B = \mathbb{0}$ or $B$ has full row rank in \eqref{eq:one-step-linear}. If $B=\mathbb{0}$, the system is autonomous and already has the \CAP{} structure. If $B$ is full row
rank, the existence of a 2-step input-affine predictor implies
\begin{equation}
\label{eq:affine-recursion-identity}
    \Phi(\Phi(x)+Bu)=\tilde{\Phi}(x)+\tilde{B}u,
    \quad \forall x\in\mathcal{X},\ u\in\mathcal{U},
\end{equation}
for a suitable function $\tilde{\Phi}$ and constant matrix $\tilde{B}$. We can take $f=\Phi$, $g_1=\Phi$, $C_1=B$,
$g_2=\tilde{\Phi}$, and $C_2=\tilde{B}$ for \Cref{lemma:func-affine-rel} with no second argument. Since the mapping $\Phi(x) + Bu$ is surjective (see \Cref{assum:one-step-predictor}), the image set $\bar{\mathcal{X}}$ in \Cref{lemma:func-affine-rel} is exactly
$\mathcal{X}$, which is open and convex. Then, \Cref{lemma:func-affine-rel} guarantees that $\Phi$ must be affine on $\mathcal{X}$ and in the form $\Phi(x)=Cx+v$ for a constant matrix $C$ and vector $v$. Consequently, the dynamics \eqref{eq:one-step-linear} take the form 
\[
x^+=Cx+Bu+v,
\]
which is a special case of the \CAP{} structure in \eqref{eq:CAP-structure}.

\textbf{Case II (Nonlinearity concentration):} We discuss the case when $\mathrm{rank}(B) = r$ and $1\le r \le n-1$. There exists an invertible change of coordinates $\tilde{x}:= Tx$ such that, in the transformed coordinates (with a slight abuse of notation), the dynamics can be written as
\begin{equation}
\label{eq:rank1-form}\begin{bmatrix}\tilde{x}_1^+\\\tilde{x}_2^+\end{bmatrix}
    =\begin{bmatrix}
        \Phi_1(\tilde{x}_1,\tilde{x}_2)\\
        \Phi_2(\tilde{x}_1,\tilde{x}_2)
    \end{bmatrix}
    +\begin{bmatrix}B_1\\\mathbb{0}\end{bmatrix}u,
\end{equation}
where $\tilde{x}_1\in\mathbb{R}^r$, $\tilde{x}_2\in\mathbb{R}^{n-r}$, $\col(\Phi_1(\tilde{x}),\Phi_2(\tilde{x}))=T\Phi(T^{-1}\tilde{x})$ and $B_1 \in \mathbb{R}^{r\times m}$ has full row rank. The transformed state space
$\tilde{\mathcal{X}}:=T\mathcal{X}$ remains open and convex, and the
transformed dynamics is surjective onto $\tilde{\mathcal{X}}$.

Since the system admits a 2-step input-affine predictor, the associated two-layer compositions are affine in $u$. For suitable functions $\tilde{\Phi}_1, \tilde{\Phi}_2$ and constant matrices $\tilde{B}_1, \tilde{B}_2$, we have
\begin{equation}
\label{eq:partial-affine-composition}
\begin{aligned}
\Phi_1(\Phi_1(\tilde{x})+B_1u,\Phi_2(\tilde{x})) & =\tilde{\Phi}_1(\tilde{x})+\tilde{B}_1u, \\
\Phi_2(\Phi_1(\tilde{x})+B_1u,\Phi_2(\tilde{x})) & =\tilde{\Phi}_2(\tilde{x})+\tilde{B}_2u,
\end{aligned}
\end{equation}
for all $\tilde{x}\in\tilde{\mathcal{X}}$ and $u\in\mathcal{U}$. We
apply \Cref{lemma:func-affine-rel} to each identity in
\eqref{eq:partial-affine-composition} by choosing $g_1=\Phi_1$, $l_1=\Phi_2$, $C_1=B_1$, and $f=\Phi_i$, $g_2=\tilde{\Phi}_i$, $C_2=\tilde{B}_i$ for $i =1,2$. The image set in \Cref{lemma:func-affine-rel} is $\tilde{\mathcal{X}}$ by surjectivity. Therefore, both
$\Phi_1$ and $\Phi_2$ are affine in the input-driven substate $\tilde{x}_1$, and \eqref{eq:rank1-form} becomes
\begin{equation}
\label{eq:CAP-general-1}
\begin{bmatrix}\tilde{x}_1^+\\\tilde{x}_2^+\end{bmatrix}
    =\begin{bmatrix}
        C_1\tilde{x}_1+g_1(\tilde{x}_2)\\
        C_2\tilde{x}_1+g_2(\tilde{x}_2)
\end{bmatrix}+\begin{bmatrix}B_1\\\mathbb{0}\end{bmatrix}u,
\end{equation}
where $C_1\in\mathbb{R}^{r\times r}$ and
$C_2\in\mathbb{R}^{(n-r)\times r}$ are constant matrices, and
$g_1,g_2$ are suitable functions.

The key feature of \eqref{eq:CAP-general-1} is that all its nonlinearities depend only on the substate $\tilde{x}_2$, whose dimension is strictly smaller than $n$ (\emph{i.e.}, $n-r$). If $C_2=\mathbb{0}$, then $\tilde{x}_2^+=g_2(\tilde{x}_2)$ is autonomous, and the dynamics have the \CAP{} structure. When $C_2\ne\mathbb{0}$, the input may still affect $\tilde{x}_2$ indirectly through $\tilde{x}_1$ so that structural constraints on the nonlinearity related to $\tilde{x}_2$ can then be imposed by multi-step input-affine predictability. Thus, by applying \Cref{lemma:func-affine-rel} recursively, we can further concentrate the nonlinearities into a smaller substate.} 

\vspace{-2mm}
\subsection{$N$-step input-affine predictors for nonlinear systems}
\label{subsec:general-nonlinear}

We now prove \Cref{them:multi-steps}, and the argument follows the same core ideas as in \Cref{sec:low-dimensional}. Our proof employs a recursive affine-composition strategy that progressively enforces structural constraints under multi-step input-affine predictability. Since the state dimension is finite, this recursive procedure terminates after at most $n-1$ steps, yielding the \CAP{} structure.

Let us consider a dynamical system of the form 
\begin{equation}
\label{eqn:nonlinear-concen}
\begin{aligned}
    x_1^+ &= M_1 u + C_1 x_{1:k-1} + g_1(x_k), \\
    x_i^+ &= M_i x_{i-1} + C_i x_{i:k-1} + g_i(x_k),  \\
    x_{k}^+ & =  M_k x_{k-1} + g_k(x_k),
\end{aligned}
\end{equation}
where the second equation applies to $i = 2, \ldots, k-1$, $\col(x_1, \ldots, x_k) := x$ with $x_i \in \mathbb{R}^{n_i}$ for $i = 1, \ldots, k$, and $M_i$ has full row rank for $i =1, \ldots, k-1$. All its nonlinearities come from the sub-state $x_k$. The system \Cref{eq:CAP-general-1}, which has a 2-step input-affine predictor, is a special case of \cref{eqn:nonlinear-concen} with $k = 2$. We may also view the one-step linear predictor \Cref{eq:one-step-linear} as a special case of \cref{eqn:nonlinear-concen} with $k = 1$. Indeed, it is not difficult to see that the system \cref{eqn:nonlinear-concen} admits $N$-step input-affine predictions with $N \in \{1, \ldots, k\}$.

If system \cref{eqn:nonlinear-concen} further admits a $(k+1)$-step input-affine predictor, then we can reveal a further structural property, depending on the rank of the coefficient $M_k$ in the last state~$x_k$.

\vspace{-2mm}

\begin{proposition}
\label{proposition:k+1-linear-predictor}
    Consider system \Cref{eqn:nonlinear-concen}. Suppose it admits a $(k+1)$-step input-affine predictor. The following statements hold.
    \begin{enumerate}
        \item If the coefficient matrix $M_k$ is $\mathbb{0}$ or has full row rank, then the system \Cref{eqn:nonlinear-concen} has a \CAP{} structure.
        \item If $M_k$ has row rank $0<r_k < n_k$, then there exists an invertible matrix $\bar{T} \in \mathbb{R}^{n_k \times n_k}$ such that with a coordinate change $\col(\bar{x}_1, \ldots, \bar{x}_{k+1}):= \col([I_{n-n_k} \ \mathbb{0}], [\mathbb{0} \ \bar{T}]) x$, the system dynamics become
\vspace{-1mm}
{\color{black}\begin{equation}
\label{eqn:nonlinear-concen-decomp}
\begin{aligned}
    \bar{x}_1^+ &= \bar{M}_1 u + \bar{C}_1 \bar{x}_{1:k} + \bar{g}_1(\bar{x}_{k+1}), \\
    \bar{x}_i^+ &= \bar{M}_i \bar{x}_{i-1} + \bar{C}_i \bar{x}_{i:k} + \bar{g}_i(\bar{x}_{k+1}),  \\
    \bar{x}_{k+1}^+ & =  \bar{M}_{k+1} \bar{x}_{k} + \bar{g}_{k+1}(\bar{x}_{k+1}),
\end{aligned}
\vspace{-1mm}
\end{equation}}
where the second equation applies to $i=2,\ldots,k$, $\bar{M}_i$ has full row rank for $i =1,\ldots, k$, $\bar{C}_i$ and $\bar{g}_i$ are suitable matrices and functions for $i =1,\ldots, k+1$.  
    \end{enumerate}
\end{proposition}
\vspace{-1mm}

The proof follows arguments similar to \cref{sec:low-dimensional} based on \Cref{lemma:func-affine-rel}. We defer the technical details to \Cref{appendix:k+1-linear-predictor}. 

From \eqref{eqn:nonlinear-concen} to \eqref{eqn:nonlinear-concen-decomp}, we essentially decompose the state $x_k$ into two components $\bar{x}_k$ and $\bar{x}_{k+1}$, and the nonlinearity is only introduced by $\bar{x}_{k+1}$ in the new coordinate. Note that the dimension of $\bar{x}_{k+1}$ is strictly smaller than that of $x_k$ so that the nonlinearity is concentrated in a smaller state. This reduction is the key proof step for \Cref{them:multi-steps}. 

\textbf{Proof of \Cref{them:multi-steps}:} 
    \textcolor{black}{The ``if'' direction follows directly from the \CAP{} structure, as the input never enters the nonlinear components, yielding input-affine predictors for all horizons. We only need to establish the ``only if'' direction.} 
    
    Since the form $x^+ = \Phi(x) + Bu$ is a specific case of \eqref{eqn:nonlinear-concen} (\emph{i.e.}, $k=1$) and it admits $N$-step linear predictors with $N=\{1,\ldots,n+1\}$, we can apply \Cref{proposition:k+1-linear-predictor} $\bar{n}$ times until the coefficient for the linear term in the nonlinear generator is $\mathbb{0}$ or full row rank. Specifically, by iteratively using statement $2$ in \Cref{proposition:k+1-linear-predictor}, there exists a series of matrices $T_0, \ldots, T_{\bar{n}-1}$, such that the dynamic system
    \vspace{-1mm}
    \begin{equation}
\label{eqn:nonlinear-cap}
\begin{aligned}
    \hat{x}_1^+ &= M_1 u + C_1 \hat{x}_{1:\bar{n}} + g_1(\hat{x}_{\bar{n}+1}), \\
    \hat{x}_i^+ &= M_i \hat{x}_{i-1} + C_i \hat{x}_{i:\bar{n}} + g_i(\hat{x}_{\bar{n}+1}),  \\
    \hat{x}_{\bar{n}+1}^+ & =  M_{\bar{n}+1} \hat{x}_{\bar{n}} + g_{\bar{n}+1}(\hat{x}_{\bar{n}+1}),
\end{aligned}
\vspace{-1mm}
\end{equation}
where the second equation applies to $i=2,\ldots, \bar{n}$, $\col(\hat{x}_1,$ $\ldots,\hat{x}_{\bar{n}+1}) = \prod_{i = 0}^{\bar{n}-1}T_i x$, satisfies $M_i$ is full row rank for $i \!=\! 1,\ldots, \bar{n}$ and $M_{\bar{n}+1}$ is $\mathbb{0}$ or full row rank. We note that $\bar{n}+1 \!\le\! n$ is guaranteed as each application of \Cref{proposition:k+1-linear-predictor} strictly reduces the dimension of the nonlinear generator.

The system \eqref{eqn:nonlinear-cap} admits an $(\bar{n}+2)$-step input-affine predictor because it admits $N$-step input-affine predictors with $N$ up to $(n+1)$ and $\bar{n}+1 \le n$. As $M_{\bar{n}+1}$ is $\mathbb{0}$ or full row rank, we can apply statement~$1$ in \Cref{proposition:k+1-linear-predictor}, which implies \Cref{eqn:nonlinear-cap} is in \CAP{} structure. This completes the proof.
\hfill $\square$

\section{Koopman Linear Embedding: Existence and Finite-dimensional Characterization}
\label{sec:accurate-Koopman}

In this section, we prove \Cref{them:accur-Koop-function,them:accur-Koop}, which characterize the existence of Koopman linear embeddings in general and in finite-dimensional function spaces, respectively. We finally discuss the stabilizability of finite-dimensional Koopman linear embeddings and Koopman LQR. 

{\color{black}

\subsection{Proof of \Cref{them:accur-Koop-function}}
\label{subsection:koop-function-sp}

The necessity of the \CAP{} structure follows from \Cref{prop:Koopman-function-infty,them:multi-steps}. We now prove sufficiency by constructing a suitable function space and linear map satisfying \eqref{eqn:Koopman-function-sp}.

Let $\tilde{x}:= Tx$ be the new coordinate with invertible transformation $T$, in which the system takes the \CAP{} form~\eqref{eq:CAP-structure}. Denote $\tilde{\mathcal X}:=T\mathcal X$ and $\mathcal X_2:=\operatorname{proj}_{\tilde x_2}\tilde{\mathcal X}$. Forward invariance of the system implies
$h(\mathcal X_2)\subseteq\mathcal X_2$, so the operator
$(\mathcal K_h\gamma)(\tilde x_2):=\gamma(h(\tilde x_2))$ is well defined for all functions $\gamma$ on $\mathcal X_2$. Writing $\tilde x_{2,i}$ for the $i$th coordinate function and $g_j$ for the $j$th component of $g$, define
\[
\begin{aligned}
    \mathcal W:=\operatorname{span}\bigl(
    \{1\}
    &\cup\{\mathcal K_h^\ell\tilde x_{2,i}:
        \ell\geq0,\ i=1,\ldots,n_2\}\\
    &\cup\{\mathcal K_h^\ell g_j:
        \ell\geq0,\ j=1,\ldots,n_1\}
    \bigr).
\end{aligned}
\]
By construction, $\mathcal K_h\mathcal W\subseteq\mathcal W$. Consider the function space
\[
\begin{aligned}
    \tilde{\mathcal V}:=\bigl\{
    \psi:\tilde{\mathcal X}\to\mathbb R \;\big|\; \psi(\tilde x)=a^\tr\tilde x_1+\eta(\tilde x_2),  a\in\mathbb R^{n_1},\ \eta\in\mathcal W
    \bigr\}.
\end{aligned}
\]

Let $\tilde{\mathcal K}_u$ denote the Koopman operator of the transformed system. For any $\psi(\tilde x)=a^\tr\tilde x_1+\eta(\tilde x_2) \in\tilde{\mathcal V}$ and $u\in\mathcal U$, the \CAP{} structure gives
\begin{equation}
\begin{aligned}
    (\tilde{\mathcal K}_u\psi)(\tilde x)
    &=a^\tr(C\tilde x_1+g(\tilde x_2)+Du)
      +\eta(h(\tilde x_2))\\
    &=(C^\tr a)^\tr\tilde x_1
      +(a^\tr g+\mathcal K_h\eta)(\tilde x_2)
      +a^\tr Du.
\end{aligned}
\label{eq:fs_cap_koopman_action}
\end{equation}
Since $a^\tr g+\mathcal K_h\eta\in\mathcal W$ and $1\in\mathcal W$, this identity implies $\tilde{\mathcal K}_u\psi\in\tilde{\mathcal V}$ for every $u\in\mathcal U$.

We next define the linear map associated with the input. The coefficient $a$ in the representation
$\psi(\tilde x)=a^\tr\tilde x_1+\eta(\tilde x_2)$ is uniquely determined by $\psi$. Indeed, if two such representations with $a_1$ and $a_2$ define the same function, the openness of $\tilde{\mathcal X}$ allows sufficiently small
perturbations $\delta$ of $\tilde x_1$ with $\tilde x_2$ fixed, yielding
$
    (a_1-a_2)^\tr\delta=0
$ 
for all such $\delta$. Thus, $a_1=a_2$, and
\[
    \tilde{\mathcal B}\psi:=D^\tr a
\]
is a well-defined linear map $\tilde{\mathcal B}:\tilde{\mathcal V}\to\mathbb R^m$.
Fix any reference input $\bar u\in\mathcal U$. By \eqref{eq:fs_cap_koopman_action}, we have
\[
    (\tilde{\mathcal K}_u\psi)(\tilde x)
    -(\tilde{\mathcal K}_{\bar u}\psi)(\tilde x)
    =\langle\tilde{\mathcal B}\psi,u-\bar u\rangle,
\]
for all $\psi\in\tilde{\mathcal V}$, $\tilde x\in\tilde{\mathcal X}$, and $u\in\mathcal U$.

Finally, transfer this construction to the original coordinates by defining
$
    \mathcal V:=\{\psi\circ T:\psi\in\tilde{\mathcal V}\},
    \, 
    \mathcal B(\psi\circ T):=\tilde{\mathcal B}\psi.
$ 
Since $T$ is bijective, $\mathcal B$ is well-defined and linear. Moreover, $x=T^{-1}\tilde x$ implies that $\mathcal V$ contains all original coordinate functions. Thus, $\mathcal V$ and $\mathcal B$ satisfy \eqref{eqn:Koopman-function-sp}, which completes the proof.}

\vspace{-3mm}
\subsection{Properties of finite-dimensional Koopman linear models} \label{subsection:properties}
Before proving \Cref{them:accur-Koop}, we present two useful structural properties of finite-dimensional Koopman linear embeddings. In particular, if a nonlinear system admits a finite-dimensional Koopman linear embedding, then it also admits another embedding via a suitable change of coordinates that satisfies the following additional properties:
\begin{enumerate}
    \item the lifted linear system is observable, and  
\item The lifting function
$\Psi:\mathcal X\to\mathbb R^{n_z}$
can be chosen to separate the original state
coordinates from the remaining observables,
\emph{i.e.}, $\Psi(x):= \col(x,\bar{\psi}(x))$.
\end{enumerate}

We first establish that the existence of a finite-dimensional Koopman linear embedding guarantees the existence of an observable one.
\vspace{-1mm}
\begin{lemma} 
\label{lem:koop-ob}
    Consider system~\cref{eqn:nonlinear} and suppose it admits a Koopman linear embedding with lifting function
$\Psi:\mathcal{X}\to\mathbb{R}^{n_z}$, 
    \begin{equation}
    \label{eqn:accu-Koop-nob}
    \Psi(f(x, u)) = A_\K \Psi(x) + B_\K u, \quad x = C_\K \Psi(x) ,
    \end{equation}
    where $(A_\K, C_\K)$ is not observable. 
Then, it also admits an \emph{observable} Koopman linear embedding with lifting function
$\bar{\Psi}:\mathcal{X}\to\mathbb{R}^{\bar n_z}$,
    \begin{equation}
    \label{eqn:accu-Koop-ob}
    \bar{\Psi}(f(x,u)) = \bar{A}_\K \bar{\Psi}(x) + \bar{B}_\K u, \quad x = \bar{C}_\K \bar{\Psi}(x) ,
     \end{equation}
    where $(\bar{A}_\K, \bar{C}_\K)$ is observable. 
\end{lemma}
\vspace{-3mm}
\begin{proof}
    The proof is constructive and relies on the Kalman observability decomposition. 
    Let $z:=\Psi(x)$ denote the lifted state. Since $(A_\K,C_\K)$ is not observable, there exists an invertible matrix $T \in \mathbb{R}^{n_z\times n_z}$ such that, under the new coordinate $\col(\hat{z}_1, \hat{z}_2):= Tz$, the lifted linear system becomes
\[
\begin{bmatrix}
    \hat{z}_1^+ \\ \hat{z}_{2}^+
\end{bmatrix}  = \begin{bmatrix}
    \hat{A}_{11} & \mathbb{0} \\ \hat{A}_{21} & \hat{A}_{22}
\end{bmatrix}
\begin{bmatrix}
\hat{z}_1 \\ \hat{z}_{2}
\end{bmatrix} + \begin{bmatrix}
    \hat{B}_1 \\ \hat{B}_{2}
\end{bmatrix} u, \quad 
x = \begin{bmatrix}
    \hat{C}_1 & \mathbb{0}
\end{bmatrix}
\begin{bmatrix}
\hat{z}_1 \\ \hat{z}_{2}
\end{bmatrix},
\]
where 
$$\begin{bmatrix}
    \hat{A}_{11} & \mathbb{0} \\ \hat{A}_{21} & \hat{A}_{22}
\end{bmatrix}= TA_\K T^{-1}, \begin{bmatrix}
    \hat{B}_1 \\ \hat{B}_2
\end{bmatrix} = TB_\K, \begin{bmatrix}
    \hat{C}_1 & \mathbb{0}
\end{bmatrix} = C_\K T^{-1},$$ and $(\hat{A}_{11}, \hat{C}_1)$ is observable.

The state $\hat z_1$ corresponds to the observable component of the lifted state, while $\hat z_2$ is unobservable and does not affect the reconstruction of $x$. Consequently, the reduced lifted dynamics
\[
\hat z_1^+ = \hat A_{11}\hat z_1 + \hat B_1 u,
\qquad
x = \hat C_1 \hat z_1,
\]
constitutes a valid Koopman linear embedding. Defining
\[
\bar{\Psi}(x) := [I_r\;\;\mathbb{0}]\,T\Psi(x),\,
\bar A_\K := \hat A_{11},\,
\bar B_\K := \hat B_1,\,
\bar C_\K := \hat C_1,
\]
where $r=\dim(\hat z_1)$, gives the desirable \eqref{eqn:accu-Koop-ob}. 
\end{proof}

We next show that a Koopman lifting can be
transformed to explicitly include the original state.
\vspace{-1mm}
\begin{lemma} \label{lem:koop-seper}
Consider system \eqref{eqn:nonlinear}, where $\mathcal X$ has nonempty interior, and suppose it admits a finite-dimensional Koopman linear embedding with lifting function $\Psi:\mathcal X\to\mathbb R^{n_z}$ and matrices $A_K,B_K,C_K$, as in \cref{eqn:accu-Koop-nob}. Then there exists an invertible matrix $T\in\mathbb R^{n_z\times n_z}$ such that
\[
    \hat{\Psi}(x):=T\Psi(x)
    =\col(x,\hat{\psi}(x)),
\]
where $\hat{\psi}:\mathcal X\to\mathbb R^{n_z-n}$. Moreover, $\hat{\Psi}$ induces another Koopman linear embedding, which is observable if and only if the original embedding
is observable.
\end{lemma}

\begin{proof}
Since $x=C_K\Psi(x)$ for all $x\in\mathcal X$ and $\mathcal X$ has nonempty interior, $C_K$ has full row rank. We can therefore choose a matrix
$N\in\mathbb R^{(n_z-n)\times n_z}$ such that
$T:=\col(C_K, N)$ is invertible. Defining $\hat{\psi}(x):=N\Psi(x)$ gives
$
    \hat{\Psi}(x)=T\Psi(x)
    =\col(x,\hat{\psi}(x)).
$ 
Furthermore, we have
\[
    \hat{\Psi}(f(x,u))
    =\hat{A}_K\hat{\Psi}(x)+\hat{B}_Ku,
    \qquad
    x=\hat{C}_K\hat{\Psi}(x),
\]
where
$
    \hat{A}_K:=TA_KT^{-1},
    \hat{B}_K:=TB_K,
    \hat{C}_K:=C_KT^{-1}=[I_n\;\;0].
$
Since observability is preserved under invertible linear coordinate transformations \cite{chen1984linear}, this completes the proof.
\end{proof}

\vspace{-5mm}
\subsection{Proof of \Cref{them:accur-Koop}} \label{subsection:proof-Koopman}
With the results in \Cref{them:multi-steps,lem:koop-ob,lem:koop-seper}, we are now ready to prove \Cref{them:accur-Koop}. 

   \textbf{Direction 1: statement $2$ $\Rightarrow$ statement $1$}: this direction is straightforward. In particular, choosing the lifting function $\Psi := \hat{\Psi} \circ T$ with $\hat{\Psi}(\tilde{x}_1, \tilde{x}_2) := \col(\tilde{x}_1, \bar{\Psi}(\tilde{x}_2))$ gives us a Koopman linear embedding with matrices in \cref{eq:Koopman-linear-model-construction}. 
    
   \textbf{Direction 2: statement $1$ $\Rightarrow$ statement $2$}: This direction is more involved. First, the Koopman linear embedding \Cref{eqn:Koop-thm2} naturally guarantees the existence of $N$-step input-affine predictors for any $N$. Thus, \Cref{them:multi-steps} guarantees that, up to a coordinate transformation, the nonlinear system can be written into the \CAP{}~form  
    \begin{equation}
    \label{eqn:multi-thm2}
    \begin{bmatrix}
        x_1^+ \\ x_2^+
    \end{bmatrix} = \begin{bmatrix}
        g(x_2) + Cx_1 \\ h(x_2)
    \end{bmatrix} + \begin{bmatrix}
        D \\ \mathbb{0}
    \end{bmatrix} u,
    \end{equation}
    where $x_1 \in \mathbb{R}^{n_1}, x_2 \in \mathbb{R}^{n_2}$ and $n = n_1 + n_2$. In the proof below, without loss of generality, we assume the system dynamics are in the form of \cref{eqn:multi-thm2}. 

    We only need to prove that the nonlinear component in \cref{eqn:multi-thm2} satisfies the additional Koopman invariant property, stated in \Cref{eq:Koopman-closed-autonomous}. Thanks to \Cref{lem:koop-ob,lem:koop-seper}, we deduce that the Koopman linear embedding \Cref{eqn:Koop-thm2} satisfies two structural properties: 
   \begin{enumerate}
       \item the matrix pair $(A_\K, C_\K)$ is observable (\Cref{lem:koop-ob});  
       \item the Koopman lifting $\Psi$ is in the form of (\Cref{lem:koop-seper})
       \begin{equation} \label{eq:lifting-for-CAP}
       \Psi(x) := \col(x_1, x_2, \psi(x_1, x_2))
       \end{equation}
       where $\psi(x_1, x_2)$ contains nonlinear functions in $x_1$ and $x_2$. 
   \end{enumerate}

   \textbf{Step 1:} We first show that there exists an invertible coordinate transformation $\tilde{T}$ such that $\tilde{\Psi}(x):= \tilde{T}\Psi(x) = \col(x_1, x_2, \tilde{\psi}(x_2))$. Compared with $\psi(x_1, x_2)$ in \eqref{eq:lifting-for-CAP}, $\tilde{\psi}(x_2)$ does not depend on $x_1$. Indeed, consider the state evolution from the initial state to step $n_z-1$ using both \cref{eqn:Koop-thm2} and \cref{eqn:multi-thm2}. 
    From the Koopman linear model \cref{eqn:Koop-thm2}, we obtain 
    \begin{equation}
    \label{eqn:prop-koop}
    \begin{bmatrix}
        x \\ x(1) \\ \vdots \\ x(n_z-1)
    \end{bmatrix} = \mathcal{O} \begin{bmatrix}
        x_1 \\ x_2 \\ \psi(x_1, x_2)
    \end{bmatrix} + \mathcal{T} \begin{bmatrix}
        u_0 \\ u_1 \\ \vdots \\ u_{n_z-2}
    \end{bmatrix} ,
    \end{equation}
    where we have  
    \[
    \begingroup
    \setlength\arraycolsep{2pt}
\def\arraystretch{0.85} 
    \mathcal{O} \!\! =\!\! \begin{bmatrix}
        C_\K \\ C_\K A_\K \\ \vdots \\ C_\K A_\K^{n_z-1}
    \end{bmatrix}, 
    \mathcal{T} \!\! =\!\! \begin{bmatrix}
\mathbb{0} & \mathbb{0} & \cdots & \mathbb{0} \\
C_\K B_\K & \mathbb{0} & \cdots & \mathbb{0} \\
C_\K A_\K B_\K & C_\K B_\K & \cdots & \mathbb{0}\\
\vdots & \vdots  & \ddots & \vdots \\
C_\K A^{n_z\!-\!2}_\K B_\K & C_\K A^{n_z\!-\!3}_\K B_\K &  \cdots & C_\K B_\K \\ 
\end{bmatrix}.
\endgroup
    \]
    From the \CAP{} structure \cref{eqn:multi-thm2}, we have
    \begin{equation}
    \label{eqn:prop-CAP}
    \begingroup
    \setlength\arraycolsep{2pt}
\def\arraystretch{0.9} 
    \begin{bmatrix}
        x \\ x(1) \\ \vdots \\ x(n_z-1)
    \end{bmatrix} = \bar{h}(x_2) + \bar{C}x_1 +\bar{D}\begin{bmatrix}
        u_0 \\ u_1 \\ \vdots \\ u_{n_z-2}
    \end{bmatrix},
    \endgroup
    \end{equation}
   where the nonlinearities in $x_2$ are captured by $\bar{h}$, and the linear dependence on $x_1$ and the input $u$ is collected in $\bar{C}$ and $\bar{D}$ respectively. The detailed forms of $\bar{h}, \bar{C}, \bar{D}$ are not important. 
   
   Since \eqref{eqn:prop-koop} and \eqref{eqn:prop-CAP} represent the same state evolution process, we must have 
    $\bar{h}(x_2) + \bar{C} x_1 = \mathcal{O} \begin{bmatrix}
        x_1 ^\tr, x_2^\tr, \psi(x_1, x_2)^\tr
    \end{bmatrix}^\tr.
    $
   We also know that $\mathcal{O}$ has full column rank because $(A_\K, C_\K)$ is observable. Thus, we have 
    \[
    \mathcal{O}^{\dag} \bar{h}(x_2) + \mathcal{O}^\dag \bar{C}x_1 =  \begin{bmatrix}
        x_1 ^\tr, x_2^\tr, \psi(x_1, x_2)^\tr
    \end{bmatrix}^\tr ,
    \]
    implying that the nonlinear component $\psi(x_1, x_2)$ must be in the form 
    $\psi(x_1, x_2) = \tilde{\psi}(x_2) + \tilde{C}x_1$. We can then choose
    \[
    \tilde{T} := \begin{bmatrix}
        I_{n_1}  & \mathbb{0} & \mathbb{0} \\
        \mathbb{0} & I_{n_2} & \mathbb{0} \\
        -\tilde{C} & \mathbb{0} & I_{n_\z-n}
    \end{bmatrix},
    \]
so that $\tilde{\Psi}(x) = \tilde{T}\Psi(x) = \col(x_1, x_2, \tilde{\psi}(x_2))$.

\textbf{Step 2:} We next derive the property \Cref{eq:Koopman-closed-autonomous} in the new coordinate with $\tilde{A}_\K \!:=\! \tilde{T}A_\K \tilde{T}^{-1}, \tilde{B}_\K \!:=\! \tilde{T}B_\K$ and $\tilde{C}_\K \!:=\! C_\K\tilde{T}^{-1}$. Let 
$
\bar{\Psi}(x_2) \!\!:=\!\! \col(x_2, \tilde{\psi}(x_2)).
$ 
One-step system propagation gives   
    \begin{equation} \label{eq:Koopman-propagation}
    \begin{bmatrix}
        A_{11} & A_{12} \\ A_{21} & A_{22}
    \end{bmatrix}\! \begin{bmatrix}
         x_1 \\ \bar{\Psi}(x_2)
    \end{bmatrix} \!+\! \begin{bmatrix}
        B_1 \\ B_2
    \end{bmatrix}\! u \!=\!\begin{bmatrix}
     g(x_2)\! +\! Cx_1 \\ \bar{\Psi}(h(x_2))
    \end{bmatrix} \!+\! \begin{bmatrix}
        D \\ \mathbb{0} 
    \end{bmatrix}\! u,
    \end{equation}
    where we have partitioned the matrix $\tilde{A}_\K$ as $\begin{bmatrix}
        A_{11} & A_{12} \\ A_{21} & A_{22}
    \end{bmatrix}$ and $\tilde{B}_\K$ as $\col(B_1, B_2)$. 
    
    By separating \cref{eq:Koopman-propagation}, the following equalities hold for any $(x, u) \in \mathcal{X} \times \mathcal{U}$
    \begin{subequations}
        \begin{align}
            A_{11}x_1 + A_{12} \bar{\Psi}(x_2) + B_1u & = g(x_2)+ Cx_1 + Du, \label{eq:Koopman-prop-1} \\
            A_{21}x_1 + A_{22} \bar{\Psi}(x_2) + B_2u & = \bar{\Psi}(h(x_2)). \label{eq:Koopman-prop-2}
        \end{align}
    \end{subequations}
    Since both $\mathcal{X}$ and $\mathcal{U}$ have non-empty interior, we have $A_{11} = C$ and $B_1 = D$ for \eqref{eq:Koopman-prop-1}, and $A_{21} = \mathbb{0}$ and $B_2 = \mathbb{0}$ for \eqref{eq:Koopman-prop-2}. That leads to 
    $
    A_{12}\bar{\Psi}(x_2) = g(x_2)$ and $A_{22}\bar{\Psi}(x_2) = \bar{\Psi}(h(x_2))$.
    Thus, we can choose $\bar{A}_\K \!\!=\!\! A_{22}$ and $\bar{C}_\K \!\!=\!\! \begin{bmatrix}
        [I_{n_2} \ \mathbb{0}] \\  A_{12}
    \end{bmatrix}$, which leads to the same formulation as \Cref{eq:Koopman-closed-autonomous}. This completes the proof.

{\color{black}
\subsection{Stabilizability and Koopman LQR}
\label{subsec:K-LQR-stability}

As summarized in the introduction (\Cref{sec:introduction}), finite-dimensional Koopman linear embeddings enable efficient controller synthesis for nonlinear systems using linear control techniques, such as Koopman MPC \cite{korda2018linear,shang2025exponential,zhang2022robust} and Koopman LQR \cite{brunton2016koopman,shang2025exponential}. For finite-horizon Koopman MPC, feasibility depends on the state and input constraints. For infinite-horizon Koopman LQR, the existence of a stabilizing solution requires stabilizability of the lifted linear model. Here, we characterize this requirement in terms of the \CAP{} structure.  

Consider a nonlinear system \cref{eqn:nonlinear} which admits a finite-dimensional Koopman linear embedding with lifting function $\Psi:\mathcal X\to\mathbb R^{n_z}$ and matrices
$A_\K,B_\K,C_\K$ in \eqref{eq:Koopman-linear-model-construction}. We then consider the unconstrained Koopman LQR problem:
\begin{equation}
\label{eqn:K-control}
\begin{aligned}
    \min_{\{u(k)\}_{k\geq0}}\
    &\sum_{k=0}^{\infty}
    (x(k)^\tr Qx(k)+u(k)^\tr Ru(k))\\
    \mathrm{subject\ to}\
    &z(k+1)=A_\K z(k)+B_\K u(k), \\
    &x(k)=C_\K z(k),\ k\geq0, \ z(0)=\Psi(x_0),
\end{aligned}
\end{equation}
where $Q\succ0$, $R\succ0$ and $x_0$ is the initial state of the system.

We establish the observability of the Koopman linear embedding in \eqref{eq:Koopman-linear-model-construction}. For the autonomous component~\eqref{eq:Koopman-closed-autonomous}, we remove linearly dependent lifting functions and apply~\Cref{lem:koop-ob}. Then, $(\bar{A}_\K,\bar{C}_\K)$ in \eqref{eq:Koopman-closed-autonomous} is observable
and components of $\bar{\Psi}$ are linearly~independent.

\begin{lemma}
\label{prop:K-observable} Consider the Koopman linear embedding \eqref{eq:Koopman-linear-model-construction}. 
If $(\bar{A}_\K,\bar{C}_\K)$ in \eqref{eq:Koopman-closed-autonomous} is observable, $(A_\K,C_\K)$ in \eqref{eq:Koopman-linear-model-construction} is observable.
\end{lemma}

We provide the proof in \Cref{apx:K-observable}. Our key idea~is~to show $\mathrm{rank}(\mathcal{O}_\K) = n_1 + \mathrm{rank}(\bar{\mathcal{O}}_\K)$, where $\mathcal{O}_\K$ and $\bar{\mathcal{O}}_\K$~are the observability matrices for $(A_\K,C_\K)$ and $(\bar{A}_\K,\bar{C}_\K)$, respectively. With $(A_\K,C_\K)$ observable and since $Q,R\succ0$, the Koopman LQR \cref{eqn:K-control} admits a stabilizing solution if and only if $(A_\K,B_\K)$ is stabilizable \cite[Section 3.3]{anderson2007optimal}. By the block structure in \eqref{eq:Koopman-linear-model-construction}, this is equivalent to $(C,D)$ being stabilizable and $\bar{A}_\K$ being Schur stable. We have the following result. 

\begin{lemma} 
\label{prop:K-autonomous-stability}
The finite-dimensional Koopman linear embedding \eqref{eq:Koopman-closed-autonomous} of the autonomous subsystem for the \CAP{} structure \eqref{eq:CAP-structure} is Schur stable if and only if
\begin{equation}
\label{eq:K-auto-limits}
\lim_{k\rightarrow\infty}h^k(\tilde{x}_2)=\mathbb{0},
 \lim_{k\rightarrow\infty}g(h^k(\tilde{x}_2))=\mathbb{0},
    \quad\forall\tilde{x}_2\in\mathcal{X}_2.
\end{equation}
\end{lemma}
We prove \Cref{prop:K-autonomous-stability} via the equivalence of propagation for $h, g$ using the \CAP{} structure \eqref{eq:CAP-structure} and the Koopman linear embedding \eqref{eq:Koopman-closed-autonomous}. We provide detailed proofs in \Cref{apx:K-autonomous-stability}.

Combining \Cref{prop:K-observable,prop:K-autonomous-stability} directly provides the following characterization for the solvability of the Koopman LQR from the associated \CAP{} structure.

\begin{proposition}[Solvability of Koopman LQR]
\label{them:K-LQR-solvability}
Suppose the nonlinear system admits a finite-dimensional
Koopman linear embedding and has the \CAP{} structure
\eqref{eq:CAP-structure}. The Koopman LQR problem \eqref{eqn:K-control} admits a stabilizing solution if and only if $(C,D)$ is stabilizable and \eqref{eq:K-auto-limits} holds.
\end{proposition}
This result follows from standard LQR theory \cite[Sec.~3.3]{anderson2007optimal}. Stabilizability of $(C,D)$ together with \eqref{eq:K-auto-limits} is equivalent to stabilizability of $(A_\K,B_\K)$, which is necessary and sufficient for \eqref{eqn:K-control} to admit a stabilizing solution when $Q,R\succ0$ and $(A_\K,C_\K)$ is observable.}

\section{Algorithmic certification of \CAP{} structures}
\label{sec:alg-cerification}
\begin{algorithm}[t]
\caption{Certification of $\infty$-step Linear Predictor}
\label{alg:multi-check}
\begin{algorithmic}[1]
  \STATE \textbf{Input:} Nonlinear system $x^+ = \Phi(x) + Bu$, $x \in \mathbb{R}^n$;
  \STATE \textbf{Output:} binary indicator and matrix $T \in \mathbb{R}^{n \times n}$;
  \STATE {\color{black}\textbf{Initialization:} \!$\textrm{flag} \!\!=\!\! 0$, \!$\textrm{iter} \!\!=\!\! 0$, \!$\Phi_0 \!\!=\!\!\Phi$, \!$B_0 \!\!=\!\! B$, \!$\Gamma_0\!:\! \mathcal{X}\!\!\rightarrow \!\!\{\emptyset\}$;}
  \WHILE{$\textrm{flag} = 0$} 
    \STATE Compute $\hat{n} := \#\textrm{number of rows of} \ B_{\textrm{iter}}$;
    \STATE Decompose subsystem: 
    \[
    (\Phi_{\text{iter}+1}, \! B_{\text{iter}+1}, \! \Gamma_{\text{iter}+1}, \! T^*, \! \textrm{flag}) \!=\! \texttt{SCD}(\Phi_{\text{iter}}, \! B_{\text{iter}}, \! \Gamma_{\text{iter}});
    \]
    \vspace{-5mm}
    \IF{$\textrm{flag} = 0$}
        \STATE $T_{\textrm{iter}} := \col([I_{n-\hat{n}}, \mathbb{0}], [\mathbb{0}, T^*])$;
        \STATE $\textrm{iter} = \textrm{iter} + 1$;
    \ENDIF
  \ENDWHILE
  \IF{flag = 1}
    \STATE indicator = ``\texttt{Exists!}", $T = \prod_{i = 0}^{\textrm{iter}-1} T_i$;
  \ELSE
    \STATE indicator = ``\texttt{Does not exist.}", $T = \texttt{None}$;
  \ENDIF
\end{algorithmic}
\end{algorithm} 

\begin{algorithm}[t]
\caption{\small \underline{S}ubsystem \underline{C}ertification and \underline{D}ecomposition (SCD)}
\label{alg:SCD}
\begin{algorithmic}[1]
  \STATE \textbf{Input:} Subsystem $x^+ = \Phi(x) + Bz$, $x \in \mathbb{R}^{\hat{n}}$ and residual nonlinearity $\Gamma(x)$
  \STATE \textbf{Output:} $\Phi^*, B^*, \Gamma^*, T^*, \textrm{flag}$;
  \STATE \textbf{Initialization:} $\Phi^*, B^*, \Gamma^*, T^* = \texttt{None}$, $\textrm{flag} = 0$;
  \IF{$B$ is full row rank}
  \IF{$\Phi(x)$ and $\Gamma(x)$ are affine on $x$}
    \STATE $\textrm{flag} = 1$; \quad $\triangleright$ The $\infty$-step linear prediction exists
  \ELSE
    \STATE $\textrm{flag} = 2$; \quad $\triangleright$ No $\infty$-step linear prediction
  \ENDIF
  \ELSIF{$B = \mathbb{0}$}
    \STATE $\textrm{flag} = 1$; \quad $\triangleright$ The $\infty$-step linear prediction exists
  \ELSE
  \STATE Find an invertible matrix $\bar{T}$, such that $\col(\bar{D}, \mathbb{0}) := \bar{T}B$, where $\bar{D}$ has full row rank; and compute 
    \[
    \begin{aligned}
       \begin{bmatrix} x_1^+ \\ x_2^+ \end{bmatrix}
       & = \bar{T} x^+  =  \bar{T}\Phi(\bar{T}^{-1}\col(x_1, x_2))\! +\! \bar{T} B z \\
       &\!=\! \begin{bmatrix} f_1(x_1,x_2)\! +\! \bar{D} z \\ f_2(x_1,x_2) \end{bmatrix};
    \end{aligned}
    \vspace{-2mm}
    \]
  \STATE Transform $\Gamma(x)$ into the new coordinate using $\bar{T}$: 
  $$
  \bar{\Gamma}(x_1, x_2) := \Gamma(\bar{T}^{-1}\col(x_1, x_2));
  \vspace{-4mm}
  $$
  \IF{$f_1, f_2, \bar{\Gamma}$ are affine on $x_1$}
  \STATE Obtain affine forms: $f_1(x_1, x_2) \!\!=\!\! \Phi_1(x_2)\! +\! B_1 x_1,$ $f_2(x_1, x_2)\! =\! \Phi_2(x_2)\!+\! B_2 x_1, 
  \bar{\Gamma}(x_1, x_2) \!=\! \gamma(x_2) \!+\! Cx_1$;
  \vspace{-3mm}
  \STATE Separate nonlinearity: 
  $\Phi^* = \Phi_2, B^* = B_2, \Gamma^* = \col(\gamma, \Phi_1), T^* = \bar{T};$ 
  \ELSE
  \STATE $\textrm{flag} = 2$; \quad $\triangleright$ No $\infty$-step linear prediction
  \ENDIF
  \ENDIF
\end{algorithmic}
\end{algorithm}

We have established that $\infty$-step input-affine prediction is equivalent to the \CAP{} structure, up to a suitable coordinate transformation (\Cref{them:multi-steps}). While this result provides a structural characterization, it does not by itself reveal how to identify such a transformation, even when a system admits a \CAP{} representation. \textcolor{black}{Moreover, constructing a transformation into the \CAP{} structure, or certifying that no such transformation exists, provides a criterion for examining the existence of a Koopman linear embedding, as shown in
\Cref{them:accur-Koop,them:accur-Koop-function}.}

In this section, we develop a symbolic algorithm to verify whether a given nonlinear system can be transformed into the \CAP{} structure and, when possible, to explicitly compute the associated invertible transformation.

\vspace{-3mm}
\subsection{Iterative Algorithm for CAP Structures Certification}

Since the proof of \Cref{them:multi-steps} in \Cref{subsec:general-nonlinear} is constructive, it naturally leads to a symbolic verification procedure. We begin with a nonlinear system of the form \cref{eqn:nonlinear-concen} with $k=1$, corresponding to a one-step input-affine predictor. At each iteration, we examine the rank of $M_k$
 and apply~\Cref{proposition:k+1-linear-predictor}. This either certifies the existence of a \CAP{} structure~or~further concentrates the nonlinear component into a lower-dimensional state. As established in the proof of \Cref{them:multi-steps}, the iteration terminates after at most $n$ steps consisting of $n-1$ nonlinearity concentration steps (statement~2 in \Cref{proposition:k+1-linear-predictor}) followed by a verification step for \CAP{} (statement 1 in \Cref{proposition:k+1-linear-predictor}).

The overall procedure is summarized in \Cref{alg:multi-check}, which relies on a key subroutine, \underline{S}ubsystem \underline{C}ertification and \underline{D}ecomposition (\texttt{SCD}), presented in \Cref{alg:SCD}. The \texttt{SCD} routine serves as a concrete instantiation of \Cref{proposition:k+1-linear-predictor} and is repeatedly invoked within the main algorithm. We note that the invertible matrix $\bar{T}$ computed in Step~13 of \Cref{alg:SCD} is generally not unique and can be obtained using standard numerical techniques such as singular value decomposition (SVD) or Gauss–Jordan elimination \cite{trefethen2022numerical} (see \Cref{lemma:sys-decompose} in \Cref{appendix:k+1-linear-predictor}). Moreover, the affineness of a function with respect to selected variables can be verified by checking whether the corresponding Jacobian exists and is constant. We have implemented \Cref{alg:multi-check,alg:SCD} in Python \footnote{See our code at  
 \url{https://github.com/soc-ucsd/Koopman-linear-embedding}.}.

We have the following result for \Cref{alg:multi-check}.
\vspace{-1mm}
\begin{proposition}[Termination of \Cref{alg:multi-check}]
\label{prop:alg-terminate}
Algorithm~\ref{alg:multi-check} terminates in at most $n$ iterations.
If the nonlinear system admits $\infty$-step input-affine prediction, then
the algorithm returns an invertible transformation that yields a CAP structure; otherwise, it verifies that no $\infty$-step input-affine prediction~exists.
\end{proposition}

\subsection{An illustrative example} 
We here revisit the discrete-time nonlinear system in \Cref{example:concise}, where we obtained a Koopman linear embedding. While the construction implicitly followed the logic of \Cref{alg:multi-check,alg:SCD}, the algorithms were not made explicit. We now apply them to the same system to illustrate the verification process. More examples are provided in \Cref{apx:extra-examples}.
    
Consider the nonlinear system $x^+ = \Phi(x) + Bu$, where we have $x \in \mathbb{R}^3, u \in \mathbb{R}$,  $B := [1, -2, 2]^\tr$ and 
    \[
    \Phi(x) := 
    \begin{bmatrix}
        (x_2+x_3)^2 + (x_1+x_2) \\
        (x_2+x_3)^2\cdot(x_2+x_3-1)+x_1 \\
        (x_2+x_3)^2\cdot(1-x_2-x_3)-x_1+0.5x_2+0.5x_3
    \end{bmatrix}.
    \]
    This system can be transformed into the \CAP{} structure after three iterations of the while-loop in \Cref{alg:multi-check}. 

    \textbf{Iteration 1:} Since $B$ is not of full row rank or zero, following steps $12$ to $14$ in \Cref{alg:SCD}, we change the coordinate to separate states of the system into 1) those fully controlled by $u$ and 2) those not controlled by $u$. Let $\bar{x} = \col(\bar{x}_1, \bar{x}_2, \bar{x}_3) := \bar{T}_0x$ be the new coordinate and $\bar{T}_0:=\col([1 \ 0 \ 0], [2 \ 1 \ 0], [-2 \ 0 \ 1])$ be the invertible transformation matrix. We have 
    \[
    \bar{x}^+ = \bar{T}_0 \Phi(\bar{T}_0^{-1}\bar{x}) + \bar{T}_0 B u 
    = \begin{bmatrix}
        f_1(\bar{x}_1, \col(\bar{x}_2, \bar{x}_3))+\bar{D}u \\
        f_2(\bar{x}_1, \col(\bar{x}_2, \bar{x}_3))
    \end{bmatrix}
    \]
    and we can write the separated form explicitly as $\bar{x}_1^+ = f_1(\bar{x}_1, \col(\bar{x}_2, \bar{x}_3))+\bar{D}u := (\bar{x}_2 + \bar{x}_3)^2 + \bar{x}_2 - \bar{x}_1 + u,$ and 
    \[
    \begin{aligned}
    \begin{bmatrix}
        \bar{x}_2^+ \\ \bar{x}_3^+
    \end{bmatrix} & = f_2(\bar{x}_1, \col(\bar{x}_2, \bar{x}_3)) \\
    & := \begin{bmatrix}
        (\bar{x}_2\!+\! \bar{x}_3)^2\cdot(\bar{x}_2\!+\!\bar{x}_3\!+\!1)\!+\!2\bar{x}_2\!-\!\bar{x}_1 \\
        (\bar{x}_2\!+\! \bar{x}_3)^2\cdot(-\bar{x}_2\!-\!\bar{x}_3\!-\!1)\!-\!1.5\bar{x}_2\!+\!0.5\bar{x}_3\!+\!\bar{x}_1
    \end{bmatrix}.
    \end{aligned}
    \]
    
    As both functions $f_1$ and $f_2$ are affine on $\bar{x}_1$, following steps $15$ to $17$ in \Cref{alg:SCD}, we have 
    \[
    \begin{aligned}
    \Phi^* \!&=\! \Phi_2 \Big(\!\begin{bmatrix}
        \bar{x}_2 \\ \bar{x}_3
    \end{bmatrix}\! \Big)\! :=\! \begin{bmatrix}
        (\bar{x}_2\!+\! \bar{x}_3)^2\cdot(\bar{x}_2\!+\!\bar{x}_3\!+\!1)\!+\!2\bar{x}_2 \\
        \!(\bar{x}_2\!+\! \bar{x}_3)^2\cdot(-\bar{x}_2\!-\!\bar{x}_3\!-\!1)\!-\!1.5\bar{x}_2\!+\!0.5\bar{x}_3
    \!\end{bmatrix}, \\
    \Gamma^*\! &= \!\Phi_1 \Big(\!\begin{bmatrix}
        \bar{x}_2 \\ \bar{x}_3
    \end{bmatrix}\!\Big) \!:=\! (\bar{x}_2 + \bar{x}_3)^2 + \bar{x}_2, \ B^* = B_2 := [-1 \ 1]^\tr.
    \end{aligned}
    \]
    From step $8$ in \Cref{alg:multi-check}, we have $T_0 := \bar{T}_0$ as $\hat{n} = n$ for the first iteration. 
    
     \textbf{Iteration 2:} We then consider the propagation of the subsystem with the state $\col(\bar{x}_2, \bar{x}_3)$ and take $\bar{x}_1$ as the external ``input". Similarly, as $B_2$ is not of full row rank or zero, we come to step $12$ in \Cref{alg:SCD}. With $\bar{T}_1:=\col([-1 \ 0], [1 \ 1])$ and $\tilde{x} = \col(\tilde{x}_2, \tilde{x}_3) := \bar{T}_1 \col(\bar{x}_2, \bar{x}_3)$, the coordinate transformed subsystem becomes
    \[
    \tilde{x}^+ = \begin{bmatrix}
        \tilde{x}_2^+ \\ \tilde{x}_3^+
    \end{bmatrix} = 
    \begin{bmatrix}
        -\tilde{x}_3^2 - \tilde{x}_3^3+2\tilde{x}_2 + \bar{x}_1 \\
        0.5 \tilde{x}_3
    \end{bmatrix}.
    \]
    In the new coordinate, the transformed nonlinearity $\bar{\Gamma}$ becomes
    \[
    \bar{\Gamma}(\tilde{x}_2, \tilde{x}_3) = \Gamma(\bar{T}_1^{-1}\col(\tilde{x}_2, \tilde{x}_3)) = \tilde{x}_3^2 -\tilde{x}_2.
    \]
   The propagation of $\tilde{x}_2, \tilde{x}_3$ and the function $\bar{\Gamma}$ are affine on $\tilde{x}_2$. Thus, following steps $16$ to $18$ in \Cref{alg:SCD}, we have 
    \[
    \Phi^* = 0.5 \tilde{x}_3, \quad B^* = 0, \quad \Gamma^* = \col(\tilde{x}_3^2, -\tilde{x}_3^2-\tilde{x}_3^3).
    \]
    From step~$8$ in \Cref{alg:multi-check}, we have $T_1 \!:=\! \col([1 \ \mathbb{0}], [\mathbb{0} \ \bar{T}_1])$ $=\col([1 \ 0\ 0], [0 \ -1 \ 0], [0 \ 1 \ 1])$.  
    
     \textbf{Iteration 3:} Since $B=0$ from the iteration 2, \Cref{alg:SCD} comes to step $10$ and we have $\textrm{flag} = 1$. \Cref{alg:multi-check} terminates, and the \CAP{} structure is identified successfully. 

     \textbf{$\infty$-step linear prediction:} Combining $T_0$ and $T_1$, we obtain the transformation $T$ in \Cref{them:multi-steps}, that is 
    \[
    T = T_1 \cdot T_0 = \begin{bmatrix}
        1 & 0 & 0 \\-2 & -1 & 0 \\ 0 & 1 & 1
    \end{bmatrix}.
    \]
    The transformed nonlinear system can be written as  
    \[
    \begin{bmatrix}
        \col(\tilde{x}_1^{+}, \tilde{x}_2^{+}) \\ \tilde{x}_3^{+} 
    \end{bmatrix} = \begin{bmatrix}
        g(\tilde{x}_3) + C \col(\tilde{x}_1, \tilde{x}_2) +D u \\
        h(\tilde{x}_3)
    \end{bmatrix},
    \]
    where we have $\tilde{x}\! :=\! \col(\tilde{x}_1, \tilde{x}_2, \tilde{x}_3) \!:=\! Tx$, $h(\tilde{x}_3) \!:=\! 0.5 \tilde{x}_3$ and  
    \[
    g(\tilde{x}_3) := \begin{bmatrix}
        \tilde{x}_3^{2} \\ -\tilde{x}_3^{2}-\tilde{x}_3^{3}
    \end{bmatrix}, C := \begin{bmatrix}
        -1 & -1  \\ 1 & 2
    \end{bmatrix} , D := \begin{bmatrix}
        1 \\ 0
    \end{bmatrix}.
    \] 
    It is clear that all nonlinearity for the propagation of states $\tilde{x}_1$ and $\tilde{x}_2$ comes from $\tilde{x}_3$, which evolves autonomously.

\section{Conclusion}
\label{sec:conclusion}
This work has provided sufficient and necessary conditions for a discrete-time controlled nonlinear system to admit a Koopman linear embedding in both infinite-dimensional and finite-dimensional settings. The key structure we established is the control-affine preserved (\CAP{}) form, which isolates the nonlinear dynamics into a self-evolving autonomous subsystem. The existence of a Koopman linear embedding is equivalent to the transformability of the \CAP{} structure. The finite-dimensional requirement further enforces the autonomous nonlinear subsystem to admit a Koopman invariant subspace that captures the nonlinearity of the full system. Interesting future directions include developing a data-driven verification process for the \CAP{} structure, constructing approximate \CAP{} structures, and quantifying how deviations from the \CAP{} structure affect the predictive performance of approximate Koopman embeddings. \textcolor{black}{We are also interested in developing necessary and sufficient conditions for characterizing nonlinear systems admitting Koopman linear embedding with nonlinear state or output recovery.}
\appendix

\subsection{Finite-dimensional model of Koopman linear embedding}
\label{apx:finite-dimensional-realization}

We show that, when the function space $\mathcal V$ in \Cref{def:Koopman-linear} is finite-dimensional, any basis of $\mathcal V$ yields a representation of the form \eqref{eqn:lifted-state}.

Let $n_\z:=\dim\mathcal V<\infty$ and choose
$\Psi:=\col(\psi_1,\ldots,\psi_{n_\z})$ which is a basis of $\mathcal{V}$.
By invariance under $\mathcal K_{\bar u}$ and
\eqref{eqn:Koopman-function-sp}, there exist
constant matrices $M$ and $B_0$ such that
\[
     \Psi(f(x,u))
    =M\Psi(x)+B_0(u-\bar u),
    \quad
    \forall (x,u)\in\mathcal X\times\mathcal U,
\]
where the $i$th row of $B_0$ is
$(\mathcal B\psi_i)^\tr$.

If $1\in\mathcal V$, there exists
$c\in\mathbb R^{n_\z}$ such that
$c^\tr\Psi(x)=1$ for all $x\in\mathcal X$.
We can then choose 
\[
    A:=M-B_0\bar u c^\tr,
    \quad B:=B_0,
\]
and the linear dynamics in \eqref{eqn:lifted-state} hold.

If $1\notin\mathcal V$, invariance under the Koopman operator implies, for every $u \in \mathcal{U}$, each component of the function
\[
B_0(u-\bar{u}) = (\mathcal{K}_u\Psi)(x)-(\mathcal{K}_{\bar{u}}\Psi)(x)
\]
is in $\mathcal{V}$. Since $\mathcal V$ contains no nonzero constant functions, $B_0(u-\bar u)=\mathbb{0}$ for all $u \in \mathcal{U}$. Thus, by choosing
$A:=M$ and $B:=\mathbb{0}$, the linear dynamics in \eqref{eqn:lifted-state} is satisfied.

Finally, since the coordinate functions belong
to $\mathcal V$, there exists a constant matrix
$C$ such that $x=C\Psi(x)$ for all
$x\in\mathcal X$. This completes the proof.

\subsection{Relationship between $N$-step input-affine predictor with predictors of shorter horizons}
\label{apx:N-step-lower}
{\color{black}The existence of an $N$-step input-affine predictor does not imply the admission of $k$-step input-affine predictors for $k<N$ in general. The nonlinearity may be eliminated through the system propagation. Consider the system 
\begin{equation}
\label{eq:counter-exp}
\begin{bmatrix}
    x_1^+ \\ x_2^+ 
\end{bmatrix} = \begin{bmatrix}
    u + x_2 u^2 \\ 0
\end{bmatrix}
\end{equation}
The nonlinear system \eqref{eq:counter-exp} does not admit a one-step linear predictor but has a 2-step linear predictor, that is $x_1(2) = u_1$ and $x_2(2) = 0$. 

When system dynamics are surjective, the existence of $N$-step input-affine predictor implies the admission of $k$-step input-affine predictors for all $k< N$, as established next.
\begin{proposition}
    Consider the nonlinear system \eqref{eqn:nonlinear}, with $f:\mathcal{X} \times \mathcal{U} \rightarrow \mathcal{X}$ surjective. If the system admits an $N$-step input-affine predictor, it has $k$-step input-affine predictors for all $k<N$.
\end{proposition}
\begin{proof}
Let $x(N) := f_N(x, u_{0:N-1}):=f(\cdots f(f(x,u_0),u_1)$ $\cdots, u_{N-1})$. We present the $k$-step linear predictor can be constructed from the $N$-step linear predictor $x(N) = \Phi_N(x) + B_Nu_{0:N-1}$. Since the nonlinear system is surjective, given an arbitrary $y \in \mathcal{X}$, there exists $x_0$ and $\bar{u} \in \mathcal{U}^{N-k}$ such that $y = f_{N-k}(x_0, \bar{u})$. Fix $v \in \mathcal{U}^k$, for any $\tilde{u} \in \mathcal{U}^k$, we have 
\[
\begin{aligned}
& \ f_k(y, \tilde{u}) -f_k(y, v) \! =\! f_N(x_0, \col(\bar{u},\tilde{u})) \! -\! f_N(x_0, \col(\bar{u},v)) \\
& \hspace{28.5mm} \!=\! \tilde{B} (\tilde{u}-v)  \\
\Rightarrow & \ f_{k}(y, \tilde{u}) = f_{k}(y,v) - \tilde{B} v + \tilde{B} \tilde{u},
\end{aligned}
\]
where $[\bar{B}, \tilde{B}] := B_N$. Thus, we can let $\Phi_{k}(x) := f_k(x, v) -\tilde{B}v$ and $B_{k}:= \tilde{B}$. This completes the proof.
\end{proof}}

\subsection{Proof of the closure of $N$-step input-affine predictor}
\label{appendix:linear-approximation}
We here illustrate if a nonlinear system fails to admit an $N$-step input affine predictors, then its finite-horizon evolution cannot be approximated arbitrarily well, uniformly over states and inputs, by any linear model. 

\begin{proposition}
\label{proposition:linear-approximation}
    \textcolor{black}{Consider the nonlinear system \eqref{eqn:nonlinear} with state domain $\mathcal{X}\subseteq\mathbb{R}^n$ and input domain $\mathcal{U} \subseteq \mathbb{R}^m$ satisfying $f(\mathcal{X}\times\mathcal{U})\subseteq\mathcal{X}$.} Denote its system propagation at step $N$ as $x(N) := {f}_N(x, u_{0:N-1})$, where $x$ is the initial state and $u_{0:N-1}$ is the control sequence. Suppose that its input set $\mathcal{U}$ is compact and has non-empty interior. If, for every \( \epsilon > 0 \), there exist a function \( \Phi_\epsilon : \mathcal{X} \to \mathbb{R}^n \) and a matrix \( B_\epsilon \in \mathbb{R}^{n \times mN} \) (independent of \( x \)) such that
\[
\sup_{x \in \mathcal{X},\, u_{0:N-1} \in \mathcal{U}^N} \left\| {f}_N(x,u_{0:N-1}) - \Phi_\epsilon(x) - B_\epsilon u_{0:N-1} \right\|_2 \leq \epsilon,
\]
then, the system necessarily admits an $N$-step linear predictor. 
\end{proposition}
\begin{proof}
\textcolor{black}{We first fix an $x\in\mathcal{X}$ and show
that $f_N(x,u_{0:N-1})$ is affine in $u_{0:N-1}$. We then present that the corresponding input coefficient
matrix is independent of $x$.}

We denote $\bar{u} := u_{0:N-1}$ for notational simplicity. Fix an arbitrary \( x \in \mathcal{X} \) and define \( g_x(\bar{u}) := f_N(x,\bar{u}) \). By hypothesis, there exists a sequence of functions $g_k(\bar{u}) := \phi_k + B_k \bar{u}$, with $\phi_k \in \mathbb{R}^n$ and $B_k \in \mathbb{R}^{n \times mN}$, such that
\[
\sup_{\bar{u} \in \mathcal{U}^N} \left\| g_x(\bar{u}) - g_k(\bar{u}) \right\|_2 \leq \frac{1}{k}. 
\]
That is, the function \( g_x(\bar{u}) \) is \textit{uniformly} approximated over \( \mathcal{U}^N \) by affine functions $\{g_k\}_{k=1}^{\infty}$. 
The limit function $g_x(\bar{u})$ must also be affine in \( \bar{u} \) \cite[Theorem 1.21]{rudin1974functional}, and hence there exist \( \Phi(x) \in \mathbb{R}^n \) and \( B(x) \in \mathbb{R}^{n \times mN} \) such that
\[
f_N(x,\bar{u}) = \Phi(x) + B(x) \bar{u}.
\]

We now show that \( B(x) \) is in fact independent of \( x \). Fix an arbitrary $x \in \mathcal{X}$, for any $\epsilon > 0$, we have 
\begin{equation}
\label{eqn:uniform-converge-u}
\|\Phi(x) + B(x) \bar{u} - \Phi_\epsilon(x) - B_\epsilon \bar{u}\|_2 \le \epsilon, \quad \forall \bar{u} \in \mathcal{U}^N.
\end{equation}
Let $C_{1\epsilon} := \Phi(x) - \Phi_\epsilon(x)$ and $C_{2\epsilon} := B(x) - B_\epsilon$. Since $\mathcal{U}^N$ has nonempty interior, we can choose $mN+1$ affinely independent input sequences  $\bar{u}_1, \ldots, \bar{u}_{mN+1} \in \mathcal{U}^N$ and denote $M := [\col(1, \bar{u}_1), \ldots, \col(1, \bar{u}_{mN+1})]$ which is invertible \cite[Box 7A]{lovasz2009matching}. We then have 
\[
\begin{aligned}
&\|\begin{bmatrix}
    C_{1\epsilon} \! \! \! &\!\! C_{2\epsilon}
\end{bmatrix} \! M \|_F\! \\
=&  \sqrt{\|C_{1\epsilon} \!+\! C_{2\epsilon}\bar{u}_1 \|_2^2 \!+\! \cdots \!+\! \|C_{1\epsilon} \!+\! C_{2\epsilon}\bar{u}_{mN+1}\|_2^2}\! \le \! \sqrt{mN\!+\!1} \epsilon,
\end{aligned} 
\]
where the inequality comes from \eqref{eqn:uniform-converge-u}. Thus, we can construct a sequence of $\{[C_{1\epsilon_k}, C_{2\epsilon_k}]\}_{k=1}^{\infty}$ with $\epsilon_k := 1/k$ such that $[C_{1\epsilon_k}, C_{2\epsilon_k}] M \rightarrow \mathbb{0}$ which implies $[C_{1\epsilon_k}, C_{2\epsilon_k}] \rightarrow \mathbb{0}$ as $M$ is invertible. That means $B_{\epsilon_k} \rightarrow B(x)$ as $C_{2\epsilon_k} \rightarrow \mathbb{0}$.

The sequence $\{B_{\epsilon_k}\}_{k=1}^{\infty}$ is independent of~$x$. Thus, its limit equals $B(x)$, which must be a constant matrix. 
\end{proof}

\subsection{Proof of \Cref{proposition:k+1-linear-predictor}}
\label{appendix:k+1-linear-predictor}

We first introduce a basic result that makes the coefficient matrix of the input full row rank.
\begin{lemma}
    \label{lemma:sys-decompose}
    Consider a dynamic system $x^+ = f(x) + C z$, where $x \in \mathbb{R}^{n_1}$, $z \in \mathbb{R}^{n_2}$ and $C \in \mathbb{R}^{n_1 \times n_2}$. There exists an invertible matrix $T$ such that the transformed system with $\col(x_1, x_2) := Tx$ has dynamics of the form
    \[
    \begin{aligned}
        x_1^+ & = f_1(x_1, x_2) + Dz, \\
        x_2^+ & = f_2(x_1, x_2),
    \end{aligned}
    \]
    where $D$ has full row rank.
\end{lemma} 
\begin{proof}
    We can use the singular value decomposition (SVD) to decompose $C$ as $C = U \Sigma V^\tr$, where $U \in \mathbb{R}^{n_1 \times n_1}, \Sigma \in \mathbb{R}^{n_1 \times n_2}$ and $V \in \mathbb{R}^{n_2 \times n_2}$. Suppose the rank of the matrix $C$ is $r$. We then have 
    $
    U^\tr U \Sigma V^\tr = \Sigma V^\tr = \col(D, \mathbb{0}),
    $  
where $D \in \mathbb{R}^{r \times n_2}$ and has full row rank. We can now consider a new coordinate $\col(x_1, x_2) := U^\tr x$ which satisfies  
    \[
    \begin{aligned}
    \begin{bmatrix}
        x_1^+ \\ x_2^+
    \end{bmatrix} = U^\tr f(U \begin{bmatrix}
        x_1 \\ x_2
    \end{bmatrix}) + U^\tr C z  
     = \begin{bmatrix}
        f_1(x_1, x_2) + D z \\
        f_2(x_1, x_2)
    \end{bmatrix}.
    \end{aligned}
    \]
That completes the proof.
\end{proof}

\textit{Proof of \Cref{proposition:k+1-linear-predictor}:} Consider system \Cref{eqn:nonlinear-concen}. The case $k=1$ follows from the
two-step analysis in \Cref{sec:low-dimensional}. We therefore
consider $k\geq2$ and first establish statement~2.

\textbf{Statement 2:} {\color{black} Suppose $0<r_k:=\operatorname{rank}(M_k)<n_k$. We use \cref{lemma:sys-decompose} with the associated $\bar{T}$ to decompose $x_k$, leading to  
\begin{equation}
    \label{eqn:nonlinear-conen-sep}
    \begin{aligned}
    \bar{x}_1^+ &= M_1 u + C_1 \bar{x}_{1:k-1} + \tilde{g}_1(\bar{x}_k, \bar{x}_{k+1}), \\
\bar x_i^+
    &=M_i\bar x_{i-1}+C_i\bar x_{i:k-1}
      +\tilde g_i(\bar x_k,\bar x_{k+1}),\\
   \bar{x}_{k}^+ & =  \hat{M}_k x_{k-1} + \tilde{g}_k(\bar{x}_k, \bar{x}_{k+1}), \\
    \bar{x}_{k+1}^+ & = \tilde{g}_{k+1}(\bar{x}_k, \bar{x}_{k+1}),
\end{aligned}
\end{equation}
where the second equation applies to $i=2,\ldots,k-1$, $\bar{x} \!:=\! \col(\bar{x}_1, \ldots, \bar{x}_{k+1}) := \hat{T} x \in \bar{\mathcal{X}} \!:= \!\hat{T} \mathcal{X}$ with $\hat{T} \!:=\! \operatorname{diag}(I_{n-n_k},\bar T)$. Furthermore, $M_i$ and $\hat{M}_k$ have full row rank and $\tilde{g}_{i}(\bar{x}_k, \bar{x}_{k+1}) := g_i(\bar{T}^{-1}\col(\bar{x}_k, \bar{x}_{k+1}))$ for $i \!=\! 1,\ldots, k-1$.} Note that there are $k+1$ substates in \cref{eqn:nonlinear-conen-sep}.

We next examine the propagation of inputs through \eqref{eqn:nonlinear-conen-sep}. By induction, $\bar x_i(t)$ is independent of the input sequence for $0\leq t<i$, $i=1,\ldots,k+1$. Indeed, the input enters the first substate directly, whereas each subsequent substate depends only on its preceding substate and substates with higher indices. In particular, both $\bar x_k(t)$ and $\bar x_{k+1}(t)$ depend only on the initial state when $0\leq t<k$. Thus, all nonlinear terms evaluated in propagating the system through time $k$ depend only on the initial state. The remaining dependence on the inputs is linear with constant coefficients. Consequently, there exist functions $h_i$ and constant matrices $D_i$ such that
\begin{equation}
    \label{eqn:dyn-k-1}
   \begin{aligned}
    \bar x_i(k)
    &=h_i(\bar x)+D_i u_{0:k-i},
        &&i=1,\ldots,k-1,\\
    \bar x_k(k)
    &=h_k(\bar x)+Mu_0,\\
    \bar x_{k+1}(k)
    &=h_{k+1}(\bar x).
\end{aligned}
\end{equation}
The first influence of $u_0$ on $\bar x_k$ propagates successively through $M_1,\ldots,M_{k-1},\hat M_k$.
Hence, $M=\hat{M}_k\prod_{i=1}^{k-1}M_i$, which has full row rank as it is a product of multiple full row rank matrices. By substituting \eqref{eqn:dyn-k-1} into \eqref{eqn:nonlinear-conen-sep}, we can obtain the system state at step $k+1$, that is
\begin{equation}
    \label{eqn:dyn-k}
\begin{aligned}
    \bar x_i(k+1)
    ={}\tilde h_i(\bar x)+E_i u_{0:k} +\tilde g_i\bigl(
          \bar x_k(k),\bar x_{k+1}(k)\bigr),
\end{aligned}
\end{equation}
for $i=1,\ldots,k+1$, where $\tilde h_i$ are suitable functions and $E_i$ are constant matrices. We define $\tilde{h}_{k+1}(\bar{x}):=\mathbb{0}$ and $E_{k+1} := \mathbb{0}$.

Since the system admits a $k+1$ step input-affine predictor, there exist functions $\hat h_i$ and constant matrices $F_i$ such that
\begin{equation}
        \label{eqn:predictor-k}
    \bar x_i(k+1)=\hat h_i(\bar x)+F_i u_{0:k},
    \qquad i=1,\ldots,k+1.
\end{equation}
Using the equivalence of \eqref{eqn:dyn-k} and \eqref{eqn:predictor-k} and substituting $\bar{x}_{k}(k), \bar{x}_{k+1}(k)$ from \eqref{eqn:dyn-k-1}, we obtain
\[
\begin{aligned}
    \tilde g_i\bigl(
       h_k(\bar{x})\!+\!Mu_0,h_{k+1}(\bar{x}))\!=\!(\hat{h}_i(\bar{x})\!-\!\tilde h_i(\bar{x}))
       \!+\!(F_i\!-\!E_i)u_{0:k}.
\end{aligned}
\]
The left-hand side is independent of $u_1,\ldots,u_k$. Since $\mathcal U$ is open, varying these inputs independently shows that their coefficient blocks in $F_i-E_i$ must vanish. Thus, for suitable functions $\delta_i$ and constant matrices $G_i$,
\begin{equation}
\label{eq:prop2-affine-composition}
    \tilde g_i\bigl(
       h_k(\bar x)+Mu_0,h_{k+1}(\bar x)\bigr)
    =\delta_i(\bar x)+G_i u_0,
\end{equation}
for all $\bar x\in\bar{\mathcal X}$, $u_0\in\mathcal U$,
and $i=1,\ldots,k+1$.

To apply \Cref{lemma:func-affine-rel}, define
\[
\begin{aligned}
    \mathcal S:=\bigl\{
       (h_k(\bar x)+Mu,h_{k+1}(\bar x)): \bar x\in\bar{\mathcal X},\ u\in\mathcal U
    \bigr\}.
\end{aligned}
\]
\textcolor{black}{Let the set of state for \eqref{eqn:nonlinear-conen-sep} at time step $k$ be $\mathcal{X}_{\mathrm{r},k} := \{y \in \mathbb{R}^n \ | \ \exists \bar{x} \in \mathcal{X}_{\mathrm{r},k-1}, u \in \mathcal{U}, \ y = \hat{T} f(\hat{T}^{-1} \bar{x}, u)\}$ where $\mathcal{X}_{\mathrm{r},0} = \bar{\mathcal{X}}$. As the function $f$ (\emph{i.e.}, system dynamics) is surjective, we have $\mathcal{X}_{\mathrm{r},k} = \bar{\mathcal{X}}$ for all time steps $k \in \mathbb{N}$.} It follows that
\[
    \mathcal S
    =\operatorname{proj}_{(\bar x_k,\bar x_{k+1})}
       \bar{\mathcal X},
\]
and $\mathcal S$ is open and convex. Together with the full row rank of $M$, \eqref{eq:prop2-affine-composition} satisfies the conditions of \Cref{lemma:func-affine-rel}. Hence, there exist constant matrices $\hat C_i$ and functions $\bar g_i$ such that
\[
    \tilde{g}_i(\bar{x}_k, \bar{x}_{k+1}) = \bar{g}_i(\bar{x}_{k+1}) + \hat{C}_i \bar{x}_k, \ \text{for} \ i = 1,2,\ldots k+1.
\]
To obtain the representation \eqref{eqn:nonlinear-concen-decomp}, we can let $\bar{M}_i = M_i, \bar{C}_{i} = [C_i \ \hat{C}_i]$ for $i = 1,\ldots,k-1$, and $\bar{M}_{k} = \hat{M}_{k}$, $\bar{C}_k = \hat{C}_k$, $\bar{M}_{k+1} = \hat{C}_{k+1}$. This completes the proof of statement 2.

\textbf{Statement 1:} If $M_k=0$, then $x_k^+=g_k(x_k)$ is autonomous. The system already has a \CAP{} structure. 

Now suppose $M_k$ has full row rank. As in \cref{eqn:dyn-k-1}, we apply~the same propagation argument to the original system, leading to 
\[
    x_k(k)=h_k(x)+Pu_0,
    \quad
    P:=\prod_{i=1}^{k}M_i,
\]
where $P$ has full row rank. Similar to \Cref{eqn:dyn-k,eqn:predictor-k,eq:prop2-affine-composition}, comparing the dynamics at time $k+1$ with the assumed input-affine predictor, and using the openness
of $\mathcal U$, yields
\[
    g_i\bigl(h_k(x)+Pu_0\bigr)
    =\delta_i(x)+G_i u_0,
    \quad i=1,\ldots,k,
\]
for suitable functions $\delta_i$ and constant matrices
$G_i$. Surjectivity of $f$ further implies
$
    \bigl\{h_k(x)+Pu:
        x\in\mathcal X,\ u\in\mathcal U\bigr\}
    =\operatorname{proj}_{x_k}\mathcal X,
$ 
which is open and convex, since $\mathcal X$ is open and convex. Applying \Cref{lemma:func-affine-rel} with no second argument, we obtain constant matrices $H_i$ and vectors
$v_i$ such that
\[
    g_i(s)=H_i s+v_i,
    \quad
    s\in\operatorname{proj}_{x_k}\mathcal X,
    \quad i=1,\ldots,k.
\]
Thus, the entire system is affine, which is a special
case of the \CAP{} structure. This completes the proof of statement~1.

\vspace{-4mm}
\subsection{Proof of \Cref{prop:K-observable,prop:K-autonomous-stability}}
\subsubsection{Proof of \Cref{prop:K-observable}}
\label{apx:K-observable}
Consider the observability matrices of Koopman linear embedding for the nonlinear system and the autonomous subsystem
\[
\begin{aligned}
\mathcal{O}_\K & :=\col(C_\K,C_\K A_\K,\ldots,
        C_\K A_\K^{n_\z-1}), \\
\bar{\mathcal{O}}_\K &:=\col(\bar C_\K,\bar C_\K\bar A_\K,\ldots, \bar C_\K\bar A_\K^{\bar{n}_\z-1}).
\end{aligned}
\]
We aim to demonstrate $\mathrm{rank}(\mathcal{O}_\K) = n_1 + \mathrm{rank}(\bar{\mathcal{O}}_\K)$, which implies $\mathcal{O}_\K$ has full column rank if and only if $\bar{\mathcal{\mathcal{O}}}_\K$ is full column rank. Since $T$ is invertible, the matrix $\bar{T} := \mathrm{diag}(T,\ldots, T)$ is also invertible. Thus, we have 
\[
\mathrm{rank}(\mathcal{O}_\K) = \mathrm{rank}(\bar{T}\mathcal{O}_\K) = \mathrm{rank}(\hat{\mathcal{O}}_\K),
\]
where $\hat{\mathcal{O}}_\K = \col(P_0,
\ldots, P_{n_\z-1},W_0,\ldots, W_{n_\z-1})$ with $P_k:= [I_{n_1} \ \mathbb{0}]A_\K^k, W_k := [\mathbb{0} \ F]A_\K^k$, and the second equality comes from the matrix permutation. Let $S_{n_z}\in\mathbb R^{n_z\times n_z}$ have ones on its first subdiagonal and zeros elsewhere. We construct the invertible matrix
\[
    L:=\mathrm{diag}
    (\bar L,I_{n_2n_z}),
    \quad
    \bar L:=I_{n_1n_z}-S_{n_z}\otimes C,
\]
such that 
\[
\begin{aligned}
\tilde{{\mathcal{O}}}_\K := L\hat{\mathcal{O}}_\K & = \col(P_0, P_1-CP_0, \ldots P_{n_\z-1}-CP_{n_\z-2},  \\
& \qquad \quad \ W_0, \ldots, W_{n_\z-1}) \\
&= \begin{bmatrix}
    I_{n_1} & \mathbb{0} \\ \mathbb{0} & M
\end{bmatrix},
\end{aligned}
\]
with $M\!:=\!\col(E, E\bar{A}_\K, \ldots, E\bar{A}_\K^{n_\z-2}, F, F\bar{A}_\K, \ldots, F\bar{A}_\K^{n_\z-1})$. Since $L$ is invertible, we have
\[
\mathrm{rank}(\mathcal{O}_\K) = \mathrm{rank}(\tilde{\mathcal{O}}_\K) = n_1 + \mathrm{rank}(M).
\]
Furthermore, as $n_\z = \bar{n}_\z+n_1$ and $n_1 > 0$, $n_\z -2 \ge \bar{n}_\z -1$. Using Cayley-Hamilton theorem, we have
\[
\begin{aligned}
& \mathrm{rank}(\bar{\mathcal{O}}_\K)  \\
= &\mathrm{rank}(\col(E,E\bar{A}_\K,\ldots,E\bar{A}_\K^{\bar{n}_z-1}, F,F\bar{A}_\K,\ldots,F\bar{A}_\K^{\bar{n}_z-1})) \\
= &\mathrm{rank}(M).
\end{aligned}
\]
This completes the proof. $\hfill\square$

\subsubsection{Proof of \Cref{prop:K-autonomous-stability}}
\label{apx:K-autonomous-stability}
For the ``only if" direction, when $\bar{A}_\K$ is Schur stable, we have
\[
    \lim_{k\rightarrow\infty}\begin{bmatrix}
        h^k(\tilde{x}_2)\\g(h^k(\tilde{x}_2))
    \end{bmatrix}
=\lim_{k\rightarrow\infty}\bar{C}_\K\bar{A}_\K^k\bar{\Psi}(\tilde{x}_2) = \mathbb{0}.
\]

For the ``if direction", using \eqref{eq:K-auto-limits}, we can write
\[
\begin{aligned}
    &\lim_{k\rightarrow\infty} \bar{\mathcal{O}}_\K\bar{A}_\K^k
      \bar{\Psi}(\tilde{x}_2) \\
=&\lim_{k\rightarrow\infty}\col(h^k(\tilde{x}_2),g(h^k(\tilde{x}_2)),\ldots,\\
&\hspace{10mm}h^{k+\bar{n}_\z-1}(\tilde{x}_2),
       g(h^{k+\bar{n}_\z-1}(\tilde{x}_2)))= \mathbb{0}.
\end{aligned}
\]
Since $\bar{\mathcal{O}}_\K$ has full column rank,
$\lim_{k\rightarrow\infty} \bar{A}_\K^k\bar{\Psi}(\tilde{x}_2)= \mathbb{0}$
for every $\tilde{x}_2\in\mathcal{X}_2$. Linear
independence of the components of $\bar{\Psi}$ implies that there exist $\bar{n}_\z$ points $\tilde{x}_2^1,\ldots,\tilde{x}_2^{\bar{n}_\z}$
for which
\[
    M:=\begin{bmatrix}
      \bar{\Psi}(\tilde{x}_2^1)&\cdots&
      \bar{\Psi}(\tilde{x}_2^{\bar{n}_\z})
    \end{bmatrix}
\]
is invertible. Thus,
$
\lim_{k\rightarrow\infty}\bar{A}_\K^kM=\mathbb{0} \Rightarrow
\lim_{k\rightarrow\infty}\bar{A}_\K^k=\mathbb{0}.
$
This completes the proof. $\hfill \square$

\subsection{Additional illustrative examples}
\label{apx:extra-examples}
We present three examples to further illustrate our algorithm: a bilinear DC motor \cite{korda2018linear,haseli2023modeling}, a pendulum \cite{de2024koopman,shang2025exponential}, and a system with an input-state separable representation. None of these systems admits a Koopman linear embedding, and our algorithm correctly returns ``\texttt{Does not exist}'' in each case. For all three examples, we take $\mathcal{X}=\mathbb{R}^n$ and $\mathcal{U}=\mathbb{R}$.

\textbf{Example 1: Bilinear motor.}
Consider a discretization of the bilinear DC motor model in \cite{korda2018linear}
\[
\begin{aligned}
x_1^+
&=(1-\frac{R_aT_s}{L_a})x_1
  -\frac{k_mT_s}{L_a}x_2u+\frac{u_aT_s}{L_a},\\
x_2^+
&=(1-\frac{b_mT_s}{J})x_2
  +\frac{k_mT_s}{J}x_1u-\frac{\tau_lT_s}{J}.
\end{aligned}
\]
The input coefficient is $r(x)= \col(-\frac{k_mT_s}{L_a}x_2,\frac{k_mT_s}{J}x_1)$.
Since $r(x)$ is not constant, the system does not admit a one-step input-affine predictor of the form $x^+=\Phi(x)+Bu$ with a constant matrix $B$. It therefore fails the necessary one-step prediction condition and is ruled out before the iterative procedure in Algorithm~1.

\textbf{Example 2: Inverted pendulum.}
Consider the discrete-time pendulum
\cite{shang2025exponential}, that is
\[
x^+ = \Phi(x) + Bu = 
\begin{bmatrix}
x_1+T_sx_2\\
(1-\frac{bT_s}{ml^2})x_2
+\frac{gT_s}{l}\sin(x_1)
\end{bmatrix}
+
\begin{bmatrix}
0\\
\frac{T_s}{ml^2}
\end{bmatrix} u.
\]

\begin{itemize}
\item \textbf{Iteration 1:}
Since $B$ is neither of full row rank nor zero, we can choose $\bar{T}_0= \begin{bmatrix} 0&1\\1&0 \end{bmatrix}$ and the new coordinate becomes  $\bar{x}:=\col(\bar{x}_1,\bar{x}_2)=\bar{T}_0x$. With $\bar{D}:=\frac{T_s}{ml^2}$, the transformed dynamics are
\[
\begin{aligned}
\bar{x}_1^+
&\!=\!f_1(\bar{x}_1,\bar{x}_2)\!+\!\bar{D}u \!:=\!(1\!-\!\frac{bT_s}{ml^2})\bar{x}_1
  \!+\!\frac{gT_s}{l}\sin(\bar{x}_2)\!+\!\bar{D}u,\\
\bar{x}_2^+
&\!=\!f_2(\bar{x}_1,\bar{x}_2):=T_s\bar{x}_1+\bar{x}_2.
\end{aligned}
\]
As $f_1$ and $f_2$ are affine in $\bar{x}_1$, \texttt{SCD} in Algorithm 2 returns
\[
\Phi^*(\bar{x}_2)=\bar{x}_2, \quad 
\Gamma^*(\bar{x}_2)=\frac{gT_s}{l}\sin(\bar{x}_2), \quad B^*=T_s.
\]

\item \textbf{Iteration 2:}
We now consider the subsystem with state $\bar{x}_2$ and take $\bar{x}_1$ as the external input. Since the updated input matrix $B^*=T_s$ has full row rank, \texttt{SCD} checks whether $\Phi^*$ and $\Gamma^*$ are affine in $\bar{x}_2$. As the $\Gamma^*(\bar{x}_2)$ includes the nonlinear term $\sin(\bar{x}_2)$, the Algorithm~1 terminates. This system does not admit the \CAP{} structure. 

\end{itemize}

\textbf{Example 3: Nonlinear system admitting input-state separable form.}
Consider the system
\[
 x^+=\Phi(x)+Bu = 
\begin{bmatrix}
\frac12x_1\\
x_1+\frac12x_2\\
\frac12x_3+x_2^2
\end{bmatrix} + 
\begin{bmatrix}
1\\
0\\
0
\end{bmatrix}u.
\]
\begin{itemize}
    \item 
\textbf{Iteration 1:}
Since $B=\col(1,0,0)$ is not full row rank or zero, we can choose $\bar{T}_0=I_3$ and let $\bar{x}:= \col(\bar{x}_1, \bar{x}_2, \bar{x}_3):=\bar{T}_0x$.
With $\bar{D}=1$, the transformed dynamics are
\[
\begin{aligned}
\bar{x}_1^+
&=f_1(\bar{x}_1,\col(\bar{x}_2, \bar{x}_3))+\bar{D}u
:=\frac12\bar{x}_1+u,\\
\begin{bmatrix}
    \bar{x}_2^+ \\ \bar{x}_3^+
\end{bmatrix}
&=f_2(\bar{x}_1,\col(\bar{x}_2, \bar{x}_3)):=
\begin{bmatrix}
\bar{x}_1+\frac12\bar{x}_2\\
\frac12\bar{x}_3+\bar{x}_2^2
\end{bmatrix}.
\end{aligned}
\]
As $f_1$ and $f_2$ are affine in $\bar{x}_1$, \texttt{SCD} in Algorithm 2 returns
\[
\Phi^*(\begin{bmatrix} \bar{x}_2 \\ \bar{x}_3\end{bmatrix})
=\begin{bmatrix}
\frac12\bar{x}_2\\
\frac12\bar{x}_3+\bar{x}_2^2
\end{bmatrix}, \quad
\Gamma^*(\begin{bmatrix} \bar{x}_2 \\ \bar{x}_3\end{bmatrix})=0, \quad 
B^*=
\begin{bmatrix}
1\\
0
\end{bmatrix}.
\]

\item \textbf{Iteration 2:}
We now consider the subsystem with state
$\col(\bar{x}_2, \bar{x}_3)$ and take $\bar{x}_1$ as the external
input. Since $B^*$ is neither of full row rank nor
zero, choose $\bar{T}_1=I_2$ and let $\tilde{x}:=\col(\tilde{x}_2,\tilde{x}_3)
:=\bar{T}_1\col(\bar{x}_2,\bar{x}_3)$.
The transformed subsystem becomes
\[
\begin{aligned}
\tilde{x}_2^+
&=f_1(\tilde{x}_2,\tilde{x}_3)+\bar{x}_1
:=\frac12\tilde{x}_2+\bar{x}_1,\\
\tilde{x}_3^+
&=f_2(\tilde{x}_2,\tilde{x}_3)
:=\frac12\tilde{x}_3+\tilde{x}_2^2.
\end{aligned}
\]
However, $f_2$ is not affine in $\tilde{x}_2$ and Algorithm~1 terminates.

\end{itemize}

We can verify the nonexistence of a Koopman linear embedding for this example by deriving its input-state separable form from the Koopman control family framework. We choose the following lifting function
\[
\Psi(x) =\col(x_1,x_2,x_3,x_1^2,x_1x_2,x_2^2,1).
\]
Let $z:=\Psi(x)$, we can write
$
z^+=\mathcal{A}(u)z,\, x=Cz,
$ 
where
\[
\mathcal{A}(u)=
\begin{bmatrix}
\frac12 & 0 & 0 & 0 & 0 & 0 & u\\
1 & \frac12 & 0 & 0 & 0 & 0 & 0\\
0 & 0 & \frac12 & 0 & 0 & 1 & 0\\
u & 0 & 0 & \frac14 & 0 & 0 & u^2\\
u & \frac12u & 0 & \frac12 & \frac14 & 0 & 0\\
0 & 0 & 0 & 1 & 1 & \frac14 & 0\\
0 & 0 & 0 & 0 & 0 & 0 & 1
\end{bmatrix},
\quad
C=
\begin{bmatrix}
I_3 & \mathbb{0}
\end{bmatrix}.
\]
Thus, this system admits an exact finite-dimensional input-state separable representation that cannot be reduced to a Koopman linear embedding.

\vspace{-2mm}

\ifCLASSOPTIONcaptionsoff
  \newpage
\fi

\bibliographystyle{IEEEtran}
\bibliography{ref}

\vspace{-1cm}
\begin{IEEEbiography}[{\includegraphics[width=1in,height=1.25in,clip,keepaspectratio]{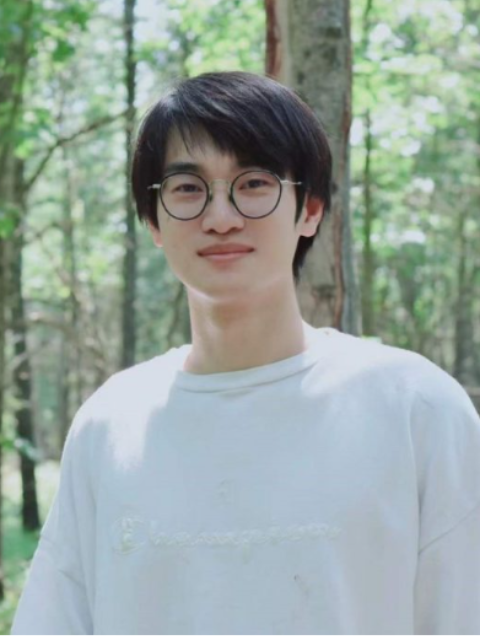}}]{Xu Shang}(Graduate Student Member, IEEE) received the B.S. degree in Mechanical Engineering from Shanghai Jiao Tong University, Shanghai, China, in 2019 and the M.S. degree in Mechanical Engineering from the University of Michigan, MI, USA in 2021. He is currently pursuing the Ph.D. degree in Electrical and Computer Engineering from the University of California San Diego, CA, USA. 

His research interests include learning, optimization, and control, with a particular focus on data-driven control and its application. He was a recipient of the Capstone Design Golden Award at Shanghai Jiao Tong University and Outstanding Graduate of Shanghai.  
\end{IEEEbiography}

\vspace{-1cm}
\begin{IEEEbiography}[{\includegraphics[width=1in,height=1.2in,clip,keepaspectratio]{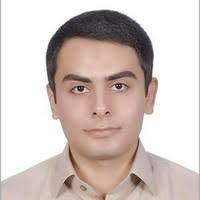}}]{Masih Haseli} received the B.S. and M.S. degrees in Electrical Engineering from Amirkabir University of Technology (Tehran Polytechnic), Tehran, Iran, in 2013 and 2015, respectively. He obtained the Ph.D. degree in Engineering Sciences (Mechanical Engineering) from the University of California, San Diego, CA, USA, in 2022. From 2022 to 2025, he held a postdoctoral position in the Department of Mechanical and Aerospace Engineering at the University of California, San Diego. He is currently a postdoctoral scholar at the California Institute of Technology, Pasadena, CA, USA. Dr. Haseli was the recipient of the 2025 IEEE Control Systems Letters Outstanding Paper Award and the Best Student Paper Award at the 2021 American Control Conference.
\end{IEEEbiography}

\vspace{-1cm}
\begin{IEEEbiography}[{\includegraphics[width=1in,height=1.2in,clip,keepaspectratio]{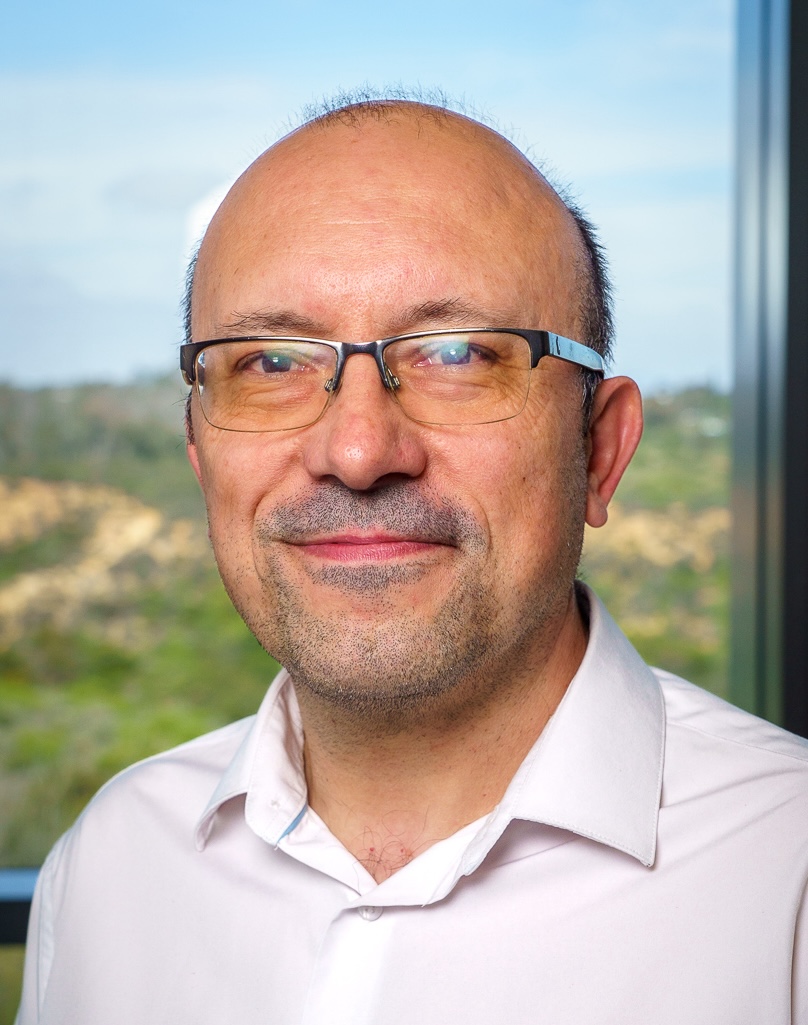}}]{Jorge
     Cort\'{e}s}(M'02, SM'06, F'14) received the Licenciatura degree in
   mathematics from Universidad de Zaragoza, Spain, in 1997, and the
   Ph.D. degree in engineering mathematics from Universidad Carlos III
  de Madrid, Spain, in 2001. He held postdoctoral positions with the
  University of Twente, Twente, The Netherlands, and the University of
  Illinois at Urbana-Champaign, Illinois, USA. 
  He is a Professor and Cymer Corporation Endowed Chair in High
  Performance Dynamic Systems Modeling and Control at the Department
  of Mechanical and Aerospace Engineering, UC San Diego, California,
  USA.  He is a Fellow of IEEE, SIAM, and IFAC.  His research
  interests include distributed control and optimization, network
  science, autonomy, learning, nonsmooth analysis, decision making under
  uncertainty, network neuroscience, and multi-agent coordination in
  robotic, power, and transportation networks.
\end{IEEEbiography}

\vspace{-1cm}
\begin{IEEEbiography}[{\includegraphics[width=1in,height=1.25in,clip,keepaspectratio]{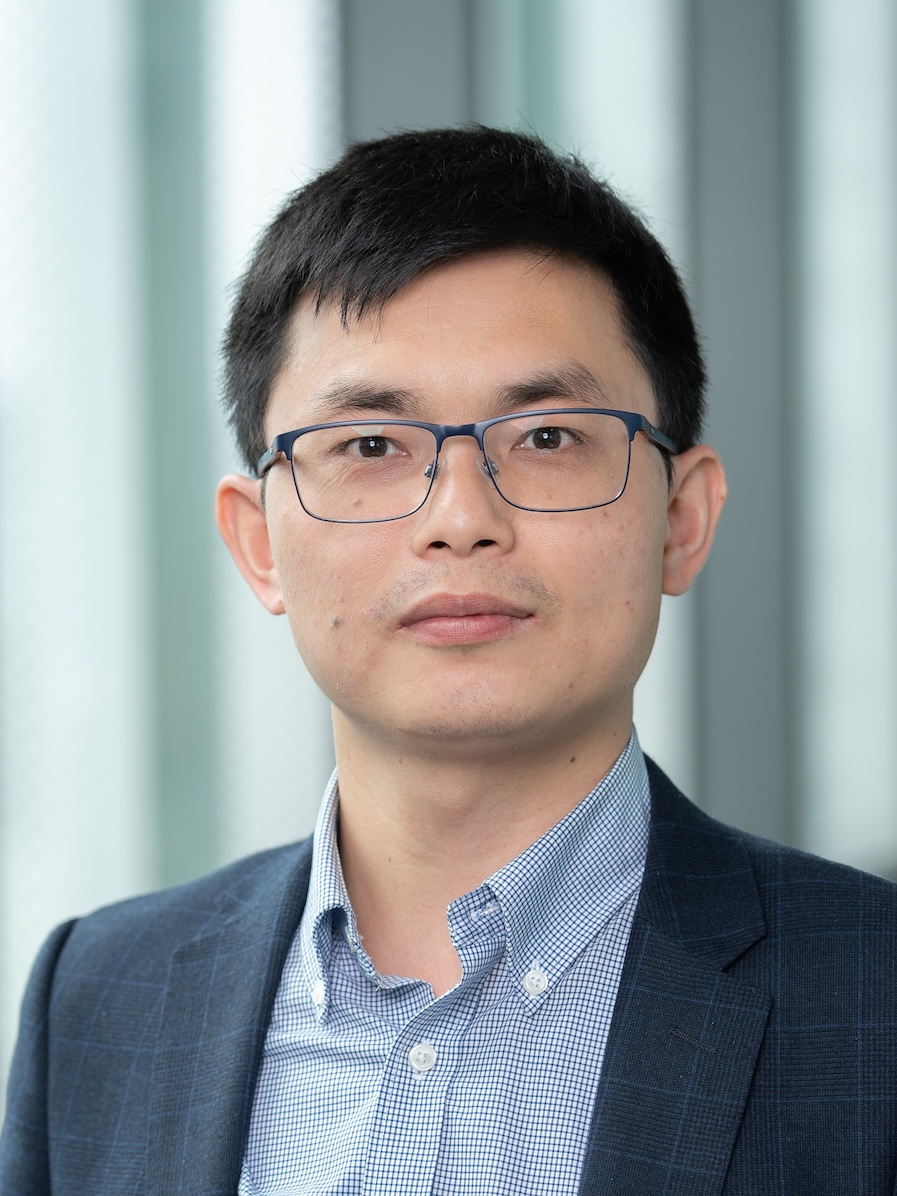}}]
{Yang Zheng} (Senior Member, IEEE) is an assistant professor in the Department of Electrical and Computer Engineering, UC San Diego.

His research interests include control theory, convex and nonconvex optimization, and traffic applications. 
Dr. Zheng was a recipient of the 2019 European Ph.D. Award on Control for Complex and Heterogeneous Systems, and the 2022 Best Paper Award in the IEEE Transactions on Control of Network Systems. He also won the 2023 Best Graduate Teacher Award from the ECE department at UC San Diego, an NSF CAREER Award in 2024 and the 2025 Donald P. Eckman Award.
\end{IEEEbiography}

\end{document}